\documentclass[11pt,xcolor=pdftex,a4paper,oneside,centertags,reqno]{amsart}

\usepackage{latexsym}
\usepackage{amsmath}
\usepackage{amsfonts}
\usepackage{amssymb}
\usepackage{graphpap}
\usepackage{epsfig}
\usepackage{amscd}
\usepackage{amsthm}
\usepackage{enumerate}
\usepackage{eucal}

\usepackage[francais]{[babel}
\usepackage[deutsch]{[babel}
\usepackage[ansinew]{inputenc}
\usepackage{graphics}

\usepackage{hyperref}
\usepackage{amsfonts}
\usepackage{pdfsync}
\usepackage{color}
\usepackage{endnotes}
\usepackage{mathrsfs}

\font\posebni=msam10
\font\tretamajka=cmbsy10 at 10pt

\font\tretamajka=cmbsy10 at 11pt

\newcommand{\4}[0]{\v{S}}

\newcommand{\C}[0]{{\mathbb C}}
\newcommand{\e}[0]{\varepsilon}
\newcommand{\f}[0]{\varphi}

\newcommand{\N}[0]{{\mathbb N}}

\newcommand{\pd}[0]{\partial}

\newcommand{\R}[0]{\mathbb{R}}

\newcommand{\leqsim}[0]{\,\text{\posebni \char46}\,}
\newcommand{\geqsim}[0]{\,\text{\posebni \char38}\,}

\newcommand{\cA}[0]{{\mathcal A}}
\newcommand{\cB}[0]{{\mathcal B}}
\newcommand{\cD}[0]{{\mathcal D}}
\newcommand{\cE}[0]{{\mathcal E}}
\newcommand{\cF}[0]{{\mathcal F}}
\newcommand{\cG}[0]{{\mathcal G}}

\newcommand{\cH}[0]{{\mathcal H}}
\newcommand{\cI}[0]{{\mathcal I}}

\newcommand{\cK}[0]{{\mathcal K}}
\newcommand{\cL}[0]{{\mathcal L}}
\newcommand{\cM}[0]{{\mathcal M}}

\newcommand{\cO}[0]{{\mathcal O}}
\newcommand{\cP}[0]{{\mathcal P}}
\newcommand{\cR}[0]{{\mathcal R}}

\newcommand{\cU}[0]{{\mathcal U}}
\newcommand{\cV}[0]{{\mathcal V}}

\newcommand\bS{\mathbf{S}}

\newcommand{\nor}[1]{\| #1 \|}

\newcommand{\Mn}[2]{\left\{ #1 : #2 \right\}}
\newcommand{\sk}[2]{\langle #1 , #2\rangle}
\newcommand{\Sk}[2]{\left\langle #1 , #2\right\rangle}

\newcommand{\wh}[0]{\widehat}

\newtheorem{theorem}{Theorem}
\newtheorem{lemma}[theorem]{Lemma}
\newtheorem{proposition}[theorem]{Proposition}
\newtheorem{corollary}[theorem]{Corollary}

\renewcommand\leq[0]{\leqslant}
\renewcommand\geq[0]{\geqslant}
\renewcommand\epsilon[0]{\varepsilon}
\renewcommand\theta[0]{\vartheta}

\newcommand\wrt{\,\text{\rm d}}
\renewcommand\mod[1]{\left\vert{#1}\right\vert}

\newcommand\norm[2]{{\left\Vert{#1}\right\Vert_{#2}}}

\newtheorem{predefinition}[theorem]{Definition}
\newenvironment{definition}%
{\begin{predefinition}\rm}{\end{predefinition}}

\newtheorem{preremark}[theorem]{Remark}  \newenvironment{remark}%
{\begin{preremark}\rm}{\end{preremark}}

\begin{document}

\title[Functional calculus ]{Functional calculus for generators of symmetric contraction semigroups}
\author[Carbonaro]{Andrea Carbonaro}
\author[Dragi\v{c}evi\'c]{Oliver Dragi\v{c}evi\'c}


\address{Andrea Carbonaro \\University of Birmingham\\ School of Mathematics\\ Watson Building\\  Edgbaston\\ Birmingham B15 2TT\\ United Kingdom}
\curraddr{Universit\`a degli Studi di Genova\\ Dipartimento di Matematica\\ Via Dodecaneso\\ 35 16146 Genova\\ Italy }
\email{carbonaro@dima.unige.it}

\address{Oliver Dragi\v{c}evi\'c \\University of Ljubljana\\ Faculty of Mathematics and Physics\\ Jadranska 21, SI-1000 Ljubljana\\ Slovenia}
\email{oliver.dragicevic@fmf.uni-lj.si}

\begin{abstract}
We prove that every generator of a symmetric contraction semigroup on a $\sigma$-finite measure space admits, for $1<p<\infty$, a H\"ormander-type holomorphic functional calculus on $L^p$ in the sector of angle $\phi^*_p=\arcsin|1-2/p|$.
The obtained angle is optimal.
\end{abstract}

\maketitle

\section{Introduction}
\label{intro}
Suppose that $(\Omega,\nu)$ is a $\sigma$-finite measure space. Let
$\cA$ be a nonnegative, possibly unbounded, self-adjoint operator on $L^2(\Omega,\nu)$,
and let $\{\cP_\lambda\}$ be the spectral resolution of the identity for which
$$
\cA f=\int_0^\infty \lambda\wrt \cP_\lambda f,\ \ \ \ f\in{\rm D}(\cA)\,.
$$
If $m$ is a bounded Borel function on $[0,\infty)$ then, by the spectral theorem, the operator
$$
m(\cA)f=\int_0^\infty m(\lambda)\wrt \cP_\lambda f,\ \ \ \ f\in L^2(\Omega,\nu)\,,
$$
is bounded on $L^2(\Omega,\nu)$. If $m(\cA)$ extends to a bounded operator on $L^p(\Omega,\nu)$ for some $p\in[1,\infty]$, we say that $m$ is a $L^p$ {\it spectral multiplier} for $\cA$. We denote by $\cM_p(\cA)$ the class of all $L^p$ spectral multipliers for $\cA$.

We shall always assume
that $\cA$ is the generator of a {\it symmetric contraction semigroup}, i.e., we always assume
that the semigroup $(\exp(-t\cA))_{t>0}$ satisfies the contraction property:
\begin{equation*}
\|\exp(-t\cA)f\|_p\leq \norm{f}{p}\,,\ \ \ \  f\in L^2(\Omega,\nu)\cap L^p(\Omega,\nu)\,,
\end{equation*}
whenever $t>0$ and $1\leq p\leq \infty$.

An interesting and challenging problem consists in characterizing the class $\cM_p(\cA)$ when $1<p<\infty$. This question is still open, even for the Euclidean Laplacian \cite{GS, HNS}. It is well known that $\cM_p(\cA)$ depends on the generator $\cA$, the measure space $(\Omega,\nu)$ and the index $p$ \cite{Mihlin1, Mihlin2, H, CS, HMM, CGHM}.

In this paper we pursue the study of ``universal'' multiplier theorems; namely, for every $p\in (1,\infty)$ we are looking for conditions on $m$ which imply that $m\in\cM_p(\cA)$ for all generators $\cA$ of symmetric contraction semigroups. The first universal multiplier theorem was proved by E.~M.~Stein, who showed that if $m$ is of Laplace transform type then $m\in \cM_p(\cA)$ for all generators $\cA$ of Markovian semigroups and every $p\in (1,\infty)$ \cite[Corollary 3, p. 121]{stein}.
Conditions on $m$ which imply that $m\in\cM_p(\cA)$ are often expressed in terms of ``regularity'' of the multiplier.

For every $\phi\in (0,\pi/2]$ consider the sector
$$
\bS_{\phi}=\{z\in\C\setminus\{0\}: |\arg z|<\phi\}\,,
$$
and denote by $H^\infty(\bS_{\phi})$ the algebra of all bounded holomorphic functions on $\bS_\phi$.
By Fatou's theorem, every $m\in H^\infty(\bS_{\phi})$ admits a nontangential limit almost everywhere on the boundary of $\bS_{\phi}$. By a slight abuse of notation, if $m\in H^\infty(\bS_{\phi})$ we use the same symbol $m$ to denote its extension to the closure $\overline\bS_{\phi}$.
We  assume that $m$ is defined at zero, so that we will be able to consider noninjective generators.

For every $J\in[0,\infty)$ we denote by $H^\infty(\bS_\phi;J)$ the class of all functions in $H^\infty(\bS_\phi)$ which satisfy the Mihlin--H\"ormander condition of order $J$ on the boundary of $\bS_\phi$. Explicitly, fix a smooth function $\psi:\R\rightarrow [0,1]$ supported in $(1/4,4)$ and such that $\psi=1$ on $[1/2,2]$. Let $H^J(\R)$ denote the usual $L^2$-Sobolev space of order $J$. Then
$$
H^\infty(\bS_\phi;J)=\{m\in H^\infty(\bS_\phi): \|m\|_{\phi;J}<\infty\},
$$
where
$$
\|m\|_{\phi;J}=\sup_{R>0}\norm{\psi\,m(e^{i\phi}R\,\cdot)}{H^J(\R)}+\sup_{R>0}\norm{\psi\,m(e^{-i\phi}R\,\cdot)}{H^J(\R)}.
$$
It is known that $H^\infty(\bS_\phi;J)$ endowed with the norm $\|\cdot\|_{\phi;J}$ is a Banach space and that $H^\infty(\bS_\phi)\subset H^\infty(\bS_\psi; J)$ as Banach spaces, for all $\phi>\psi$ and $J\geq 0$.

It has been known for a while that generators of symmetric contraction semigroups have a {\it bounded holomorphic functional calculus} on $L^p$ in sectors of angle smaller than $\pi/2$. More precisely, for every
$p\in(1,\infty)$ there exists an angle $\theta_p\in(0,\pi/2)$
such that if $\theta>\theta_p$ and $m\in H^\infty(\bS_{\theta})$, then $m\in\cM_p(\cA)$ and there exists $C(p,\theta)>0$ such that
\begin{equation}
\label{eq: dim free}
\|m(\cA)\|_p\leq C(p,\theta)\left(\|m\|_{\theta;0}+|m(0)|\right)\,.
\end{equation}
Indeed, M.~Cowling \cite{cowling} improved Stein's universal multiplier theorem by showing that
$$
\theta_p=\phi^C_p=\pi\mod{\frac1p-\frac12}
$$
suffices for all generators of symmetric contraction semigroups. In the special case when $\cA$ generates a sub-Markovian semigroup, Cowling's result has been improved by P. C.~Kunstmann and \v{Z}.~\4trkalj \cite{KS}
who showed that $\phi^C_p$ can be replaced by
$$
\phi^{K{\text{\it\4}}}_p=\frac{\pi}{2}
\mod{\frac1p-\frac12}
+\left(1-\mod{\frac1p-\frac12}
\right)\arcsin\frac{|p-2|}{2p-|p-2|}\,.
$$
Recently C. Kriegler \cite[Remark 2]{KR} extended this result to generators of symmetric contraction semigroups.

For every $p\in (1,\infty)$ define
\begin{equation}
\label{d: 1}
\phi^*_p=\arcsin \mod{1-\frac{2}{p}}\,.
\end{equation}
Note that $0<\phi^*_p< \phi^{K{\text{\it\4}}}_p<\phi^C_p$ whenever $p\neq 2$. Interestingly, $\phi^*_p$ relates to the angle of analyticity on $L^p$ of all symmetric contraction semigroups. Indeed, as a consequence of a more general result by Kriegler \cite[Theorem 5.4 and Corollary 6.2]{KR}, any symmetric contraction semigroup is contractive and holomorphic on $L^p$ in the sector $\bS_{\pi/2-\phi^*_p}$ (see Proposition~\ref{t: 1}).

It is known that for some specific generators the angle $\phi^*_p$ suffices. For example, Laplacians on Euclidean spaces and, more generally, Laplace--Beltrami operators on doubling Riemannian manifolds whose heat kernels satisfy gaussian estimates admit a H\"ormander-type functional calculus \cite{H}, \cite[Theorem 3.1]{DOS}, \cite[Theorem 7.23]{O}. Consequently these operators have a bounded holomorphic functional calculus on $L^p$ in any sector. It is also known that $\phi^*_p$ suffices for Laplace--Beltrami operators on symmetric spaces of noncompact type \cite{CS,ST,AL,A} and, more generally, on manifolds with bounded geometry and optimal $L^2$ spectral gap \cite{Ta1,MMV,Ta2}.
However, all of the results above entail estimates of the $L^p$ norms of the multipliers which depend on the dimension of the underlying space, while universal multiplier theorems give dimension-free estimates as in \eqref{eq: dim free}.

The given examples involve Riemannian manifolds where the volume of balls satisfies the doubling property at least for small radii. The problem of reducing the angle $\phi^{K{\text{\it\4}}}_p$ for Laplacians on (weighted) manifolds that are not locally doubling seems to be more difficult to solve, even with dimension-dependent estimates. Indeed, to the best of our knowledge, the only example studied and well understood in this framework is the symmetric finite-dimensional Ornstein--Uhlenbeck operator, i.e., the weighted Laplacian
\begin{equation*}
\label{eq: def ou}
\cL_{OU}=\Delta_{\R^n}+x\cdot \nabla_{\R^n}
\end{equation*}
on the Euclidean space $\R^n$ endowed with the Gaussian measure
$$
\wrt \gamma_n(x)= (2\pi)^{-n/2}e^{-|x|^2/2}\wrt x.
$$
In \cite[Theorem 1]{fcou} and \cite[Theorem 1.2]{MMS} it was proved that $\phi^*_p$ suffices for $\cL_{OU}$. This was extended in
\cite{Sasso} to the Laguerre operators and in \cite{C} to certain perturbations of $\cL_{OU}$.
The proofs in \cite{fcou,MMS} rely heavily on properties of the particular operator in question. In particular, the authors use an explicit formula
for the heat kernel with complex time for $\cL_{OU}$ (Mehler formula) and, combining it with a change of variable in time and a stationary phase argument, they prove that for every $p>1$ there exists $C(n,p)>0$ such that
\begin{equation}
\label{eq: int 1}
\|\cL^{is}_{OU}\|_{\cB\left(L^p(\R^n,\gamma_n)\right)}\leqslant C(n,p)\,e^{\phi^*_p|s|}
\end{equation}
for all $s\in\R$. By a result of Meda (see Propositions~\ref{t: 1.2} and \ref{t: 1.2 bis}), the estimate above implies that for every $J>1$ there exists $\widetilde{C}(n,p,J)>0$ such that if $m\in H^\infty(\bS_{\phi^*_p};J)$, then
\begin{equation}
\label{eq: int 2}
\|m(\cL_{OU})\|_{\cB\left(L^p(\R^n,\gamma_n)\right)}\leqslant \widetilde{C}(n,p,J)\left(\|m\|_{\phi^*_p;J}+|m(0)|\right)\,.
\end{equation}
The estimates in \eqref{eq: int 1} and \eqref{eq: int 2} depend on the dimension of the underlying space, $\R^n$. This is rather unpleasant, since it does not provide any information on the functional calculus for infinite-dimensional Ornstein--Uhlenbeck operators, i.e., Ornstein--Uhlenbeck operators on Wiener spaces.

\subsection*{Main result}
Few years ago, G.~Mauceri and S.~Meda remarked to the first named author of the present paper that $\phi^*_p$ could be the optimal angle in the universal multiplier theorem for generators of symmetric contraction semigroups. The aim of this paper is to confirm this supposition by showing
that, for every generator of a symmetric contraction semigroup, the angle $\phi^{K{\text{\it\4}}}_p$ can be replaced by $\phi^*_p$. Namely, we prove
the following result. First let us introduce the notation $p^*=\max\{p,p/(p-1)\}$ for $p>1$.

\begin{theorem}[Multiplier theorem]
\label{c: 1.1}
For every generator $\cA$ of a symmetric contraction semigroup on a $\sigma$-finite measure space $(\Omega,\nu)$, every
$p\in(1,\infty)$, $J>3/2$ and $m\in H^\infty(\bS_{\phi^*_p};J)$, the operator $m(\cA)$ extends to a bounded operator on $L^p(\Omega,\nu)$, and there exists $C(J)>0$ such that
$$
\norm{m(\cA)}{{\mathcal B}(L^p(\Omega,\nu))}\ \leq C(J)\,(p^*-1)^{9/4}\log p^*\left( \norm{m}{\phi^*_p;J}+|m(0)|\right)\,.
$$
\end{theorem}
In short, generators of symmetric contraction semigroups admit a {\it H\"ormander-type holomorphic functional calculus} of order $J>3/2$ in the sector $\bS_{\phi^*_p}$. We remark that, independently of the order of the H\"ormander condition, the angle $\phi^*_p$ in Theorem~\ref{c: 1.1} is sharp, since it is optimal for the Ornstein--Uhlenbeck operator $\cL_{OU}$; see \cite[Theorem 2]{fcou} and \cite[Theorem 2.2]{HMM}.

\subsection*{Consequences of the multiplier theorem}
Apart from improving the aforementioned results of Stein \cite{stein}, Cowling \cite{cowling} and Kunstmann and \4trkalj \cite{KS}, our multiplier theorem in particular applies to the
(finite as well as infinite-dimensional)
Ornstein--Uhlenbeck
operator 
and, because we get dimension-free estimates, it
improves multiplier theorems in \cite{fcou, MMS}.
However, the result in \cite{MMS} requires a weaker order of differentiability of the multiplier, namely $J>1$.

By combining Theorem~\ref{c: 1.1} with the techniques developed by Cowling in \cite[Section 3]{cowling} we improve his maximal theorems \cite[Theorems 7 and 8]{cowling}. The precise formulation of this statement is contained in Theorem~\ref{t: cowling imp}. This theorem, together with a standard argument of Stein \cite[pp. 72--82]{stein}, in turn leads to new results concerning pointwise convergence $\exp(-z\cA)f\rightarrow f$ as $z\rightarrow 0$ in the closed sector $\overline\bS_\phi$, for $0\leq\phi<\phi_p$; see Corollary~\ref{c: pointwise conv} for a concise statement.

By  Theorem~\ref{c: 1.1} and \cite[Corollary 6.7]{CDMY} we obtain the following square-function estimates. Let $\phi>\phi^*_p $ and $v,C>0$. Suppose that $\psi\in H^\infty(\bS_\phi)$ satisfies $|\psi(\zeta)|\leq C|\zeta|^v/(1+|\zeta|^2)^v$ for all $\zeta\in \bS_\phi$. Then
$$
\norm{\left(\int^\infty_0|\psi(t\cA)f(\cdot)|^2\frac{\wrt t}{t}\right)^{1/2}}{p}\leq C(p,\phi) \|f\|_p\,,\ \ \ \  f\in L^p(\Omega,\nu)\,.
$$

We also mention that Theorem~\ref{c: 1.1} combined with \cite[Theorem 6.5]{KW} implies that, for $1< p<\infty$ and $0<\alpha<\pi/(2\phi^*_p)$, the operator $\cA^\alpha_p$ has the so-called maximal $L^q$-regularity for all $1<q<\infty$; see, for example, \cite{KW} and \cite[Section 4]{KS}.
Moreover, by combining Theorem~\ref{c: 1.1} with \cite[Theorem 5.3]{KW}, we deduce that $\cA$ is $R$-sectorial on $L^p(\Omega,\nu)$ of angle $\phi^*_p$. We refer the reader to \cite{KW} and the references therein for a detailed discussion about Rademacher boundedness and some of its consequences.

\subsection*{Outline of the proof}
Our approach consists of proving suitable $L^p$ estimates of imaginary powers of $\cA$ (Proposition~\ref{p: 5}) which will then, through a theorem of Meda
(Propositions~\ref{t: 1.2} and \ref{t: 1.2 bis}), imply the multiplier theorem. In order to prove the estimates of the imaginary powers we first derive a bilinear embedding theorem with complex time (Theorem~\ref{t: sycn bil embedding}).

Theorem~\ref{t: sycn bil embedding} is proven by studying monotonicity properties of the heat flow associated with a particular Bellman function defined by Nazarov and Treil \cite{NT}. This reduces the problem to proving new convexity properties of the Bellman function in question (Theorem~\ref{t: convex}).

The proof of Theorem~\ref{t: sycn bil embedding} can be viewed, on one hand, as an extension of the
Bellman function technique, widely known since the mid 1990s when it was introduced in harmonic analysis by Nazarov, Treil and Volberg in the first preprint version of their paper \cite{NTV} and afterwards employed in a large number of papers, of which the closest ones to our approach are \cite{PV, NV, DV, DV-Sch, DV-Kato, CD}.
On the other hand, it may be considered as an extension of the heat flow method used by Bakry \cite{Bakry3} for proving analyticity properties of symmetric diffusion semigroups.
When studying the heat flow associated with the Bellman function, we put emphasis on encompassing complex times and obtaining, apart from the monotonicity, also quantitative estimates of the flow's derivative.

Let us, for the reader's convenience, put down a scheme summarizing the mutual dependence of the results in this paper:
$$
\left.
\begin{array}{rcccl}
\text{Bellman function properties}\\
\text{Heat flow}
\end{array}
\hskip -2pt
\right\} \hskip -1pt
 \Rightarrow \text{Bilinear embedding} \Rightarrow \hskip -3pt
\raisebox{-8.1pt}{$\left.
\begin{array}{r}
 \text{Estimates of } \|\cA^{is}\|_p  \\
 \text{Theorem of Meda}
\end{array}
\hskip -2pt \right\}
\Rightarrow \text{Multiplier theorem}$}
$$
At the end let us mention that, for all we know, Bellman functions have so far never been used for the purpose of proving general spectral multiplier theorems. We plan to further extend this method, in particular to analogue problems on UMD spaces.

\subsection*{Organization of the paper}
In Section~\ref{prep} we introduce much of the notation and invoke some of the basic objects and known results that will be used in the rest of the paper, among them the theorem of Meda about imaginary powers. In Section~\ref{s: 3} we formulate the bilinear embedding theorem (Theorem~\ref{t: sycn bil embedding}) and show that it, together with Meda's result, implies the multiplier theorem (Theorem~\ref{c: 1.1}). The rest of the paper is then devoted to establishing Theorem~\ref{t: sycn bil embedding}. The main idea utilized to that end is presented in Section~\ref{s: heat-flow}. In Section~\ref{s: bellman} we define the Bellman function of Nazarov and Treil and formulate the crucial convexity property that this function satisfies (Theorem~\ref{t: convex}). This property is verified in Section~\ref{s: convexity}. Finally, in Section~\ref{s: proof bilinear} we complete the proof of the bilinear embedding and thus of the multiplier theorem.

\section{Preparation}
\label{prep}

Given two quantities $A$ and $B$, we adopt the convention whereby $A\leqsim B$ means that there exists an absolute constant $C>0$ such that $A\leqslant CB$. If both $A\leqsim B$ and $B\leqsim A$, then we write $A\sim B$. If $\{\lambda_1,\dots,\lambda_n\}$ is a set of parameters, $C(\lambda_1,\dots,\lambda_n)$ denotes a constant depending only on $\lambda_1,\dots,\lambda_n$. When $A\leq C(\lambda_1,\dots,\lambda_n)B$, we will often write $A\leqsim_{\lambda_1,\dots,\lambda_n} B$.

If $B$ is a matrix-valued function on $\R^N$, more precisely, $B:\R^N\rightarrow\R^{M,M}$, then for $\xi\in\R^N$ and $\omega\in\R^M$ we set
\begin{equation}
\label{d: 2}
B[\xi;\omega]=\Sk{B(\xi)\omega}{\omega}_{\R^{M}}\,,
\end{equation}
where $\sk{\cdot}{\cdot}_{\R^{M}}$ denotes the usual scalar product in $\R^M$.

Throughout the paper, the hypotheses on $(\Omega,\nu)$, $\cA$ and $\exp(-t\cA)$ will be as specified early in the introduction.
In order to simplify the notation we set, for $\Re z>0$,
$$
T_z=\exp(-z\cA)\,.
$$
The semigroup $(T_t)_{t>0}$ is strongly continuous for $1\leq p<\infty$ and weak* continuous for $p=\infty$. When necessary to distinguish between different $p$, we denote by $\cA_p$ the generator of $(T_t)_{t>0}$ on $L^p(\Omega,\nu)$. Furthermore, ${\rm D}(\cA_p)$, ${\rm R}(\cA_p)$ and ${\rm N}(\cA_p)$ will stand for the domain, range and null-space of $\cA_p$, respectively. The following result is well known.

\begin{lemma}\label{l: fin prop. markov}The orthogonal projection $\cP_0$ on ${\rm N}(\cA_2)$ is given by
$$
\cP_0f=\lim_{t\rightarrow\infty}T_tf\,,\ \ \ \ f\in L^2(\Omega,\nu)\,.
$$
Moreover, for every $p\in (1,\infty)$ the following holds:
\begin{itemize}
\item[({\rm i})]
the operator $\cP_0$ extends to a contractive projection on $L^p(\Omega,\nu)$. In particular, $L^p(\Omega,\nu)=\overline{{\rm R}(\cA_p)}\oplus {\rm N}(\cA_p)$, where the sum is direct.
\item[({\rm ii})] ${\rm D}(\cA_2)\cap {\rm D}(\cA_p)$ is dense both in $L^2(\Omega,\nu)$ and $L^p(\Omega,\nu)$, and ${\rm D}(\cA_2)\cap {\rm R}(\cA_2)\cap L^p(\Omega,\nu)$ is dense in $\overline{{\rm R}(\cA_p)}$.
\end{itemize}
\end{lemma}

For every $p\in [1,\infty]$, we denote the norm on $L^p(\Omega,\nu)$ by $\|\cdot\|_p$ and set
$$
q=p'=p/(p-1).
$$
By abuse of notation, if $S$ is a bounded operator on $L^p(\Omega,\nu)$, we denote its operator norm by $\|S\|_p$. If $\varphi\in L^p(\Omega,\nu)$ and $\psi\in L^q(\Omega,\nu)$, the dual pairing between $\phi$ and $\psi$ is denoted by
$$
\sk{\varphi}{\psi}=\int_\Omega\varphi\overline\psi\,d\nu\,.
$$
For any $\phi\in (-\pi/2,\pi/2)$, we define $\phi^*=\pi/2-\phi$; with this notation, \eqref{d: 1} implies
$$
\phi_p=\phi_{q}= \arccos\mod{1-\frac2p}=\arctan \frac{2\sqrt{p-1}}{|p-2|}\,.
$$
This is a small abuse of notation, since we previously denoted $p*=\max\{p,p/(p-1)\}$ for $p>1$. We however believe that the context will always make it clear whether the asterisk is being applied to a ``Lebesgue exponent'' or to an ``angle''.\\

Next result concerns analyticity on $L^p$ of symmetric contraction semigroups, and was first proven by Bakry \cite[Th\`eor\'eme 3]{Bakry3} for a certain class of sub-Markovian semigroups which contains, for example, those generated by weighted Laplace--Beltrami operators on complete manifolds. The result was than extended by Liskevich and Perelmuter \cite[Corollary 3.2]{lp} to all sub-Markovian semigroups; see also \cite[Theorem 3.13]{O}. Proposition~\ref{t: 1} was recently obtained by Kriegler \cite[Corollary 6.2]{KR}, as a consequence of a more general result on noncommutative $L^p$ spaces.
The above-mentioned results improve an earlier estimate of this kind due to Stein \cite[III, {\tretamajka \char120}2, Theorem 1, page 67]{stein}.

\begin{proposition}
[\cite{KR}]
\label{t: 1}
For every $p\in (1,\infty)$, the semigroup $(T_t)_{t>0}$ extends to a contractive analytic semigroup on $L^p(\Omega,\nu)$ in the sector $\bS_{\phi_p}$.
\end{proposition}
\begin{remark}\label{r: 1}
By \cite[p. 72]{stein}, as a consequence of Proposition~\ref{t: 1}, if $p\in (1,\infty)$, $\phi\in (-\phi_p,\phi_p)$ and $f\in L^p(\Omega,\nu)$, then for almost every $x\in\Omega$ the function $t\mapsto(T_{te^{i\phi}}f)(x)$ is real-analytic on $(0,\infty)$. This fact will be implicitly used in the rest of this paper.
\end{remark}

If $M\in L^{\infty}(\R_{+})$, define its {\it Laplace transform} $\widetilde M$ by the rule
$$
\widetilde M(\lambda)=\lambda\int^{\infty}_{0}M(t)e^{-t\lambda}\wrt t,\ \ \ \ \Re \lambda>0\,.
$$

A prominent special case of multipliers of Laplace transform type are the imaginary powers $\cA^{is}$, $s\in \R$. It is well known that there is a deep relation between the growth of the norms of imaginary powers of $\cA$ and its holomorphic functional calculus on sectors \cite[Theorem 4]{M}, \cite[Theorem 5.4]{CDMY}. In particular, we have the following result, essentially due to Meda \cite{M} and proved in \cite[Theorem 2.1]{MMS} (for the case $J\in\N$ we refer the reader to \cite[Theorem 2.2]{fcou}).

\begin{proposition}
\label{t: 1.2}
Suppose that $1 < p < \infty$ and that there exist constants $C_0,\sigma>0$ and $\phi^*\in [0,\pi/2)$ such that
\begin{equation*}
\label{eq: impowass}
\norm{\cA^{is}f}{p}\leq C_0(1+|s|)^\sigma e^{\phi^*|s|}\norm{f}{p}\,,\hskip 40pt  \forall f\in {\rm R}(\cA_p)\,, \ \ \ \forall s\in\R\,.
\end{equation*}
If $J>\sigma+1$ and $m\in H^\infty(\bS_{\phi^*};J)$, then $m(\cA)$ extends to a bounded operator on $\overline{{\rm R}(\cA_p)}$.
\end{proposition}
Since in Theorem~\ref{c: 1.1} we shall also be interested in the behavior of the $L^p$ norm of the multiplier $m(\cA)$ with respect to $p$,
we now briefly recall the proof of Proposition~\ref{t: 1.2}, however with the distinction of tracking down the dependence on $p,\phi^*$ and $J$  wherever it appears.

For every $\tau\in(0,1)$ set
$$
m_\tau(t,\lambda)=(t\lambda)^\tau e^{-t\lambda}m(\lambda)\,, \ \ \ \  t>0,\ \ \  \lambda\in\bS_{\phi^*}\,,
$$
and denote by $s\longmapsto\mathscr M_{s} m_\tau$ the Mellin transform of $m_\tau$ with respect to the second variable. Explicitly, for $t>0$,
$$
[\mathscr M_{s}m_\tau](t)=\int_0^\infty m_\tau(t,\lambda)\lambda^{-is}\frac{\wrt \lambda}\lambda\,,\ \ \ \ s\in\R\,.
$$
\begin{lemma}\label{l: 11}
Let $\tau\in(0,1)$, $s\in\R$, $\alpha\geq0$, $\phi^*\in [0,\pi/2)$ and $m\in H^\infty(\bS_{\phi^*};\alpha)$.
Then
\begin{equation*}
\nor{\mathscr M_{s}m_\tau}_{L^\infty(\R_+)}
\leqslant
C(\alpha)\tau^{-1}(\cos\phi^*)^{-\tau-\alpha}(1+|s|)^{-\alpha}e^{-{\phi^*}|s|}\|m\|_{\phi^*;\alpha}
\end{equation*}
\end{lemma}
\begin{proof}
Suppose for a moment that $\alpha\in\N$. Then the lemma follows by modifying \cite[Theorem 4]{M} as described in the proof of \cite[Theorem 2.2]{fcou}.
In order to prove the lemma for noninteger values of $\alpha$, just observe that by a simple argument of interpolation between Sobolev spaces the conclusion of \cite[Theorem 4]{M} still holds if $\alpha\in\R_+\setminus \N$.
\end{proof}
Following the proof of \cite[Theorem 2.1]{CM} we arrive at the representation formula
\begin{equation}\label{eq: 10}
m(\cA)f=\frac{2^{\tau}}{\pi \Gamma(\tau+1)}\int^{+\infty}_{-\infty}
\widetilde{[\mathscr M_{s}m_\tau]}(\cA)\,
\cA^{is}f\wrt s\,, \ \ \ \   f\in {\rm R}(\cA_p)\,,
\end{equation}
where $\Gamma$ denotes the Euler's function.

By Cowling's multiplier theorem \cite[Theorem 3]{cowling}, Laplace transform type multipliers belong to $\cM_p(\cA)$. Denote by $C_2=C_2(p)$ the best constant in
$$
\|\widetilde{M}(\cA)f\|_p\leq C\|M\|_\infty\|f\|_p\hskip 40pt \forall f\in{\rm R}(\cA_p)\,, \forall M\in L^\infty(\R_+).
$$
We can now easily obtain the following ``quantitative'' version of Proposition \ref{t: 1.2}.
\begin{proposition}
\label{t: 1.2 bis}
Under the assumptions of Proposition \ref{t: 1.2},
$$
\norm{m(\cA)f}{p}\leq C_1\norm{f}{p}\,,\ \ \ \ f\in \overline{{\rm R}(\cA_p)}\,,
$$
where
\begin{equation}
\label{eq: best cons}
C_1=
C(J)\,C_2(p)
\inf_{\tau\in(0,1)}
\int^{+\infty}_{-\infty}
\nor{\mathscr M_{s}m_\tau}_{L^\infty(\R_+)}\,\cdot
 \|\cA^{is}\|_{\cB\left(L^p(\overline{{\rm R}(\cA_p)},\nu)\right)}\wrt s
\end{equation}
is finite.
\end{proposition}
Consequently, in order to prove Theorem~\ref{c: 1.1}, we will first estimate the constant $C_2$ (see Proposition~\ref{l: laplace}), and then we will estimate the growth in $s\in\R$ of the $L^p$ norms of imaginary powers $\cA^{is}$ (see Proposition~\ref{p: 5}). To this end, we will use the following subordination formula.

For every $\theta\in(-\pi/2,\pi/2)$, consider the positively oriented path
$$
\gamma_\theta=\Mn{re^{i\theta}}{r>0}\,.
$$
\begin{lemma}\label{l: sub. formula}
Suppose that either $\psi\in (0,\pi/2)$ and $M\in H^\infty(\bS_{\psi})$, or $\psi=0$ and $M\in L^{\infty}(\R_{+})$. Then, for every $\theta\in (-\psi,\psi)\cup \{0\}$, $f\in {\rm D}(\cA_2)\cap {\rm R}(\cA_2)$ and $g\in L^2(\Omega,\nu)$,
$$
\Sk{\widetilde M(\cA)f}{g}=2
\int_{\gamma_\theta}
\Sk{\cA T_zf}{T_{\overline{z}}g}\,
M(2z)\wrt z\,.
$$
Consequently,
$$
\mod{\Sk{\widetilde M(\cA)f}{g}}
\leq
2
\nor{M}_{L^{\infty}(\gamma_\theta)}
\int_{\gamma_\theta}
\mod{\sk{\cA T_zf}{T_{\overline{z}}g}}
\wrt|z|\,.
$$
\end{lemma}
\begin{proof}
The lemma quickly follows from the spectral theorem, the Cauchy theorem and the fact that on $L^2$ we have $T^*_z=T_{\overline{z}}$.
\end{proof}

\section{Bilinear embedding and proof of the multiplier theorem}
\label{s: 3}
We now state a complex-time bilinear embedding theorem which is the principal tool for proving Theorem~\ref{c: 1.1}.

If $p\in (2,\infty)$ and $\epsilon \in (0,1/2)$ define
\begin{equation}
\label{eq: 15}
p_\epsilon=\frac{p-2\epsilon}{1-\epsilon}\hskip 20pt \quad {\rm and}\hskip 20pt \quad q_\epsilon=p_\e'=\frac{p_\epsilon}{p_\epsilon-1}\,.
\end{equation}
\begin{theorem}[Bilinear embedding]
\label{t: sycn bil embedding}
For every $p\in(2,\infty)$, $\epsilon\in (0,1/2)$, $\phi\in [-\phi_{p_\epsilon},\phi_{p_\epsilon}]$,
and all $f\in L^p(\Omega,\nu)$, $g\in L^q(\Omega,\nu)$,
$$
\int_{\gamma_\phi}
\mod{\sk{\cA T_zf}{T_{\overline{z}}g}}
\wrt|z|
\leqslant \frac{30(p-1)}{\epsilon\cos\phi}\|f\|_p\|g\|_q\,.
$$
\end{theorem}
The proof of the bilinear embedding is deferred to Section~\ref{s: proof bilinear}. Here we prove some consequences of Theorem~\ref{t: sycn bil embedding}, most notably Theorem~\ref{c: 1.1}. We first consider multipliers of Laplace transform type and, in particular, imaginary powers.

\begin{proposition}
\label{l: laplace}
Suppose that  $M\in L^\infty(\R_{+})$. Then, for every $p\in (1,\infty)$,
$$
\nor{\widetilde M(\cA)f}_p\leq 120 (p^{*}-1)\|M\|_{\infty}\|f\|_p\,,\ \ \ \   \forall f\in \overline{{\rm R}(\cA_p)}\,.
$$
\end{proposition}
\begin{proof}
By duality it suffices to assume that $p> 2$. The lemma follows by combining Lemma~\ref{l: fin prop. markov} ({\rm ii}) with Lemma~\ref{l: sub. formula} and Theorem~\ref{t: sycn bil embedding} applied with $\theta=\phi=0$ and $\epsilon=1/2$.
\end{proof}

\begin{proposition}
\label{p: 5}
For every $s\in\R$, $p>1$ and $f\in\overline{{\rm R}(\cA_p)}$,
\begin{equation*}
\label{eq: pnorm imaginary}
\norm{\cA^{is} f}{p}\leqsim (p^*-1)
(1+|s|)^{1/2}e^{\phi^*_{p}|s|}\norm{f}{p}\,.
\end{equation*}
Moreover, for every $p^*\geq 3$ and $f\in \overline{{\rm R}(\cA_p)}$,
\begin{equation*}\label{eq: pot imm stima fine}
\|\cA^{is}f\|_p\leqsim\
e^{\phi^*_p|s|}\|f\|_p
\begin{cases}
(p^*-1)(1+|s|)^{-1/2}& {\rm if}\ \ \ |s|\leq2p^*\sqrt{p^*-1}/(p^*-2)\,;\\
\sqrt{p^*-1}(1+|s|)^{1/2}& {\rm if}\ \ \ |s|> 2p^*\sqrt{p^*-1}/(p^*-2)\,.
\end{cases}
\end{equation*}
\end{proposition}

\begin{proof}
By duality it suffices to consider $p\in(2,\infty)$. Note that $\cA^{is}=\widetilde M_s(\cA)$, $s\in\R$, where
$M_s(z)=\Gamma(1-is)^{-1}z^{-is}$ for $z\in\C$, $\Re z>0$. Evidently $M_s\in H^\infty(\bS_\theta)$ for all $\theta\in (0,\pi/2)$ and, by the asymptotic estimate $\mod{\Gamma(1-is)}\thicksim \mod{s}^{1/2}e^{-\pi\mod{s}/2}$ as $\mod{s}\rightarrow\infty$,
$$
\nor{M_s}_{L^\infty(\gamma_\theta)}
\leqsim (1+|s|)^{-1/2}e^{\frac{\pi|s|}{2}+s\theta}
$$
for all $\theta\in (-\pi/2,\pi/2)$ and $s\in\R$.

Fix $s\in\R$, $\epsilon\in (0,1/2)$, $r>2$ and $\phi\in [0,\phi_{r_\epsilon}]$.
By combining Lemma~\ref{l: sub. formula} (applied with $\theta=-\phi\,\text{sign}\,s$) with
Lemma~\ref{l: fin prop. markov} ({\rm ii}) and Theorem~\ref{t: sycn bil embedding}, we obtain
\begin{equation}
\label{eq: est imaginary}
\norm{\cA^{is} f}{r}\leqsim  \frac{r-1}{\epsilon \cos\phi}(1+|s|)^{-1/2}e^{\phi^*|s|}\norm{f}{r} \hskip 30pt f\in\overline{{\rm R}(\cA_r)}\,.
\end{equation}
At this point we need to consider two cases, when $|s|$ is large and small, respectively.
The decision on where to split is made according to the elementary inequality
\begin{equation}
\label{eq: est. phi_p epsilon}
\phi^*_{p_\epsilon}\leq \phi^*_p+\frac{2(p-2)}{p\sqrt{p-1}}\,\epsilon
\hskip 40pt \forall p>2,\ \forall \epsilon\in(0,1/2)\,.
\end{equation}

Suppose first that $|s|>2p\sqrt{p-1}/(p-2)$. Then we apply \eqref{eq: est imaginary} with $r=p$,
$$
\epsilon
=\frac{p\sqrt{p-1}}{p-2}\cdot\frac{1}{1+|s|}
$$
and $\phi=\phi_{p_\epsilon}$.
From \eqref{eq: est. phi_p epsilon} and the estimate $\cos\phi_{p_\epsilon}>\cos\phi_p=1-2/p$, we obtain
\begin{equation*}
\label{eq: 12}
\norm{\cA^{is} f}{p}\leqsim \sqrt{p-1}
(1+|s|)^{1/2}e^{\phi^*_{p}|s|}\norm{f}{p}\hskip 30pt f\in\overline{{\rm R}(\cA_p)}\,.
\end{equation*}

Suppose now that $|s|\leq 2p\sqrt{p-1}/(p-2)$. We separate further two cases:
\begin{itemize}
\item
If $2<p\leq 3$, then by interpolating the trivial $L^2$-estimate with \eqref{eq: est imaginary} applied with $r=2p$, $\phi=0$ and $\epsilon=1/3$, we obtain
\begin{equation*}
\label{eq: 13}
\norm{\cA^{is}f}{p}\leqsim (p-1)\|f\|_p \leqsim (p-1)(1+|s|)^{1/2}e^{\phi^*_{p}|s|}\norm{f}{p}\hskip 30pt  f\in \overline{{\rm R}(\cA_p)}\,.
\end{equation*}
\item
If $p\geq 3$
then $\cos\phi_{p_\epsilon}\geq\cos\phi_3= 1/3$. Therefore, by \eqref{eq: est imaginary} applied with $r=p$, $\epsilon=1/3$ and $\phi=\phi_{p_\epsilon}$, and from \eqref{eq: est. phi_p epsilon}, we obtain
$$
\|\cA^{is}f\|_p\leqsim (p-1)(1+|s|)^{-1/2}e^{\phi^*_p|s|}\|f\|_p\hskip 30pt f\in \overline{{\rm R}(\cA_p)}\,.
$$
\end{itemize}
The proposition now follows.
\end{proof}

\subsection*{Proof of Theorem~\ref{c: 1.1}.}
By combining the first inequality in Proposition~\ref{p: 5} with Proposition~\ref{t: 1.2}, applied with $\sigma=1/2$ and $\phi^*=\phi^*_p$, we obtain that $m(\cA)$ is bounded on $\overline{\cR(\cA_p)}$.

We now estimate the norm of $m(\cA)$ on $\overline{\cR(\cA_p)}$ by means of Propositions~\ref{t: 1.2 bis}. By duality we may restrict ourselves to the case $p>2$. Since $\cos\phi^*_p=2\sqrt{p-1}/p$, Lemma~\ref{l: 11} applied first with $\phi^*=\phi^*_p$ and $\alpha=0$, and then with $\phi^*=\phi^*_p$ and $\alpha=J$ implies
\begin{equation*}
\nor{\mathscr M_{s}m_\tau}_{L^\infty(\R_+)}
\ \leqsim_J\
\tau^{-1} p^{\tau/2}e^{-{\phi_p^*}|s|}
\min\left\{1,\bigg(\frac{\sqrt{p}}{1+|s|}\bigg)^J\right\}
\|m\|_{\phi_p^*;J}.
\end{equation*}
By combining the inequality above with Proposition~\ref{l: laplace}, the second estimate in Proposition~\ref{p: 5} and \eqref{eq: best cons}, we eventually obtain, for any $\tau\in (0,1)$,
$$
\|m(\cA)f\|_p\leqsim_J\tau^{-1}p^{\tau/2}p^{9/4}\|m\|_{\phi^*_p;J}\hskip 30pt f\in{\rm R}(\cA_p)\,.
$$
To conclude the proof of the theorem, just minimize the right--hand side of the inequality above with respect to $\tau\in (0,1)$ and observe that, by Lemma~\ref{l: fin prop. markov},
$$
m(\cA)f=m(0)\cP_0f+m(\cA)(I-\cP_0)f\,,\ \ \ \ f\in L^2(\Omega,\nu)\cap L^p(\Omega,\nu)\,,
$$
where $\cP_0$ is a contraction on $L^p(\Omega,\nu)$ and $(I-\cP_0)f\in \overline{{\rm R}(\cA_p)}$.
\qed

\medskip
As already remarked in the introduction, by combining our results with Cowling's subordination theory \cite[Section 3]{cowling}, we improve his maximal theorems \cite[Theorems 7,8]{cowling} as follows.
For $N\in\N$ set $P_N(\lambda)=\sum^N_{n=0}(-\lambda)^n/n!$. For every $\delta>0$, $N\in\N$  and $\phi\in [0,\pi/2)$ define the maximal operators
$$
\cM^\phi f=\sup\Mn{|T_zf|}{z\in \overline{\bS}_\phi}
\hskip 13pt
\text{and}
\hskip 13pt
\cM^\phi_{N,\delta}f=\sup\Mn{|z|^{-\delta}|T_zf-P_N(z\cA)f|}{z\in \overline{\bS}_\phi}.
$$
\begin{theorem}
\label{t: cowling imp}
Suppose that $1<p<\infty$.
\begin{itemize}
\item[({\rm i})] If $0\leq\phi< \phi_p$, then there exists $C=C(\phi,p)>0$ such that
$$
\|\cM^\phi f\|_p\leq C\|f\|_p\,,
$$
for all $f\in L^2(\Omega,\nu)\cap L^p(\Omega,\nu)$.
\item[({\rm ii})]
If $\phi=\phi_p$, and $1\leq N<\delta<N+1$, then there exists $C=C(\delta,p)>0$ such that
$$
\|\cM^\phi_{N,\delta}f\|_p\leq C \|\cA^\delta f\|_p\,,
$$
for all $f$ in $L^2(\Omega,\nu)$ such that $f$ and $\cA^\delta f$ are in $L^p(\Omega,\nu)$.
\end{itemize}
\end{theorem}
\begin{proof}
The proof is exactly as in \cite[Theorems 7,8]{cowling} except that one should instead of \cite[Corollary 1]{cowling} use Proposition~\ref{p: 5}.
\end{proof}
The theorem above, together with a standard argument of Stein \cite[pp. 72--82]{stein}, implies the following result, which is an improvement over \cite[Corollaries 2,3]{cowling}.

\begin{corollary}
\label{c: pointwise conv}
Suppose that $1<p<\infty$. If either $0\leq \phi<\phi_p$ and $f$ is in $L^p(\Omega,\nu)$, or $\phi=\phi_p$, $1\leq N< \delta<N+1$ and $f, \cA f,\dots, \cA^{N}f$ and $\cA^\delta f$ are in $L^p(\Omega,\nu)$, then $T_zf$ converges to $f$ pointwise almost everywhere as $z$ tends to $0$ in $\overline{\bS}_\phi$.
\end{corollary}

\section{The heat-flow method}
\label{s: heat-flow}

In this section we illustrate the technique we will utilize for proving the bilinear embedding in Theorem~\ref{t: sycn bil embedding}. The exposition will be rather descriptive, aimed at giving the idea of the proof without dwelling on technical details which will be addressed later.

Fix $p\in (2,\infty)$ and $Q:\R^4\rightarrow \R_+$ which is of class $C^1$. We shall say that the flow associated with $Q$ (and with $\cA$) is {\it regular on $L^p(\Omega,\nu)\times L^q(\Omega,\nu)$} if for every $f\in L^p(\Omega,\nu)$, $g\in L^q(\Omega,\nu)$ and $\phi\in (-\phi_p,\phi_p)$, the function $\cE:[0,\infty)\mapsto\R_+$, defined by
$$
\cE(t)=\int_\Omega Q\left(T_{te^{i\phi}}f, T_{te^{-i\phi}}g\right)\wrt\nu\,,
$$
is continuous on $[0,\infty)$, differentiable on $(0,\infty)$ with a continuous derivative and
\begin{equation}
\label{eq: flow der}
\cE'(t)=\int_\Omega \frac{\pd}{\pd  t}Q(T_{te^{i\phi}}f,T_{te^{-i\phi}}g)\wrt\nu\,.
\end{equation}

We are interested in finding a nonnegative function $Q\in C^1(\R^4)$, possibly depending on $p$ and $\phi$,
whose corresponding flow admits, for every generator $\cA$, the following properties:
\begin{itemize}
\item
regularity on $L^p\times L^q$;
\item
initial value bound, i.e., the existence of $A_0=A_0(\phi,p)>0$ such that
\begin{equation}
\label{eq: heatflow p2}
\cE(0)\leq A_0(\|f\|^p_p+\|g\|^q_q)\,,
\hskip 50pt \forall f\in L^p, g\in L^q;
\end{equation}
\item
``quantitative monotonicity'', i.e.,
the existence of $B_0=B_0(\phi,p)>0$ such that
\begin{equation}
\label{eq: heatflow p1}
-\cE'(t)
\geq B_0\mod{\Sk{\cA T_{te^{i\phi}}f}{T_{te^{-i\phi}}g}}\,,
\hskip 50pt \forall f\in L^p, g\in L^q, t>0.
\end{equation}
\end{itemize}
For if these conditions are fulfilled, then
$$
B_0\int^\infty_0
\mod{\Sk{\cA T_{te^{i\phi}}f}{T_{te^{-i\phi}}g}}
\wrt t\leq -\int^{\infty}_{0}\cE'(t)\wrt t\leq \cE(0)\leq A_0(\|f\|^p_p+\|g\|^q_q)\,.
$$
By replacing $f$ with $\lambda f$ and $g$ with $g/\lambda$ and minimizing with respect to $\lambda>0$ the right-hand side of the inequality above, we obtain the bilinear embedding,
$$
\int_{\gamma_\phi}\mod{\Sk{\cA T_{z}f}{T_{\bar z}g}}
\wrt |z|\leqslant C(\phi,p)\|f\|_p\|g\|_q,
$$
for all $f\in L^p(\Omega,\nu)$ and $g\in L^q(\Omega,\nu)$, where $C(\phi,p)=q(p-1)^{1/p}A_0/B_0$.\\

Let us elaborate a bit on the conditions above. Write $\zeta=(\zeta_1,\zeta_2)$, $\eta=(\eta_1,\eta_2)$ and define
$$
\partial_\zeta=\frac{1}{2}\left(\partial_{\zeta_1}-i\partial_{\zeta_2}\right)
\hskip 30pt \text{and} \hskip 30pt
\partial_\eta=\frac{1}{2}\left(\partial_{\eta_1}-i\partial_{\eta_2}\right)\,.
$$
Suppose that, for $\zeta,\eta\in\R^2$, the function $Q$ satisfies the following estimates:

\begin{equation}
\label{es: 1}
0\leq Q(\zeta,\eta)\leq A_0\left(|\zeta|^p+|\eta|^q\right)
\end{equation}
and
\begin{equation}
\label{es: 2}
\mod{(\partial_{\zeta}Q)(\zeta,\eta)}\leqsim\,|\zeta|^{p-1}+|\eta|
\hskip 30pt
\text{and}
\hskip 30pt
\mod{(\partial_{\eta}Q)(\zeta,\eta)}\leqsim\,|\eta|^{q-1}+|\zeta|\,.
\end{equation}
Then \eqref{es: 1} and \eqref{es: 2} together with Proposition~\ref{t: 1} give the regularity on $L^p\times L^q$ of the flow associated with $Q$. Moreover \eqref{es: 1} clearly implies \eqref{eq: heatflow p2}.

Now let us turn towards \eqref{eq: heatflow p1}.
We may expand \eqref{eq: flow der} as
$$
\aligned
-\cE'(t)
=2\Re\int_\Omega & \Big[e^{i\phi}\left(\partial_\zeta Q\right)(T_{te^{i\phi}}f,T_{te^{-i\phi}}g)\cA T_{te^{i\phi}}f\\
&+ e^{-i\phi}\left(\partial_\eta Q\right)(T_{te^{i\phi}}f,T_{te^{-i\phi}}g)\cA T_{te^{-i\phi}}g\Big]\wrt\nu\,.
\endaligned
$$
Thus, condition \eqref{eq: heatflow p1} can be restated as
\begin{equation}
\label{eq: heatflow p1 bis}
B_0\mod{\int_\Omega \cA f \cdot \overline{g}\wrt \nu}\leq 2\Re\int_\Omega \left(e^{i\phi}(\partial_\zeta Q)(f,g)\cA f+ e^{-i\phi}(\partial_\eta Q)(f,g)\cA g\right)\wrt\nu\,,
\end{equation}
for all $f\in {\rm D}(\cA_p)$, $g\in{\rm D}(\cA_q)$.

\subsection*{Convexity of $Q$: The Riemannian case.} We are interested in pointwise conditions on the second order partial derivatives of $Q$ which imply \eqref{eq: heatflow p1 bis}. We first consider the special case when $\nu=\mu$ is the Riemannian measure on a complete Riemannian manifold $\Omega=M$ and $\cA=\Delta$ is the nonnegative Laplace--Beltrami operator on $M$.

If $\phi\in (-\pi/2,\pi/2)$, set
\begin{equation}
\label{planinic}
\cO_\phi=\left[ {\begin{array}{rr}
  \cos\phi&-\sin\phi\\
  \sin\phi&\cos\phi
 \end{array} } \right]\in\R^{2,2}
\hskip 30pt
\text{and}
\hskip 30pt
\cU_\phi=\left[ {\begin{array}{ll}
 \cO_\phi&0\\
0&\cO_{-\phi}
 \end{array} } \right]\in\R^{4,4}\,.
\end{equation}
Suppose for a moment that $Q\in C^2(\R^4)$.
Let $\cH(Q)$ denote the Hessian matrix of $Q$.
Define
\begin{equation}
\label{eq: def of R_phi}
\cR_\phi(Q)=\frac{1}{2}
\left(
\cU_\phi^T\cdot \cH(Q)
+
\cH(Q)\cdot
\cU_\phi\right).
\end{equation}
The proof of \cite[Theorem 3.5]{S} shows that $C^\infty_c(M)$ is a core for ${\rm D}(\Delta_r)$ when $1<r<\infty$, so that it suffices to consider \eqref{eq: heatflow p1 bis} for $f,g\in C^\infty_c(M)$. Then, an integration by parts on the right-hand side of \eqref{eq: heatflow p1 bis} gives
$$
2\Re\int_{M} \left(e^{i\phi}(\partial_\zeta Q)(f,g)\Delta f+ e^{-i\phi}(\partial_\eta Q)(f,g)\Delta g\right)\wrt \mu=\int_{M}\cR_\phi(Q)\left[(f,g);(\wrt f,\wrt g)\right]\wrt \mu,
$$
where $\wrt$ denotes the differential (there is a small yet obvious abuse of notation as to the way the identity
\eqref{d: 2} was used above). We remark that a similar identity is valid also in the case when $\Delta$ is the Ornstein--Uhlenbeck operator on a  Wiener space.

It is from here not difficult to see that we could prove \eqref{eq: heatflow p1 bis} if we had the following property of $Q$:
there exist $B_0=B_0(p,\phi)>0$ and $\tau:\R^4\rightarrow (0,\infty)$ such that
for all $\xi\in\R^4$ and $\omega=(\omega_1,\omega_2)\in\R^2\times\R^2$,
$$
\cR_\phi(Q)[\xi;\omega]\geqslant B_0\left(\tau(\xi)|\omega_1|^2+\tau(\xi)^{-1}|\omega_2|^2\right)\,.
$$
Owing to a theorem by Treil, Volberg and the second author \cite[Theorem 2.1]{DTV}, the inequality above is equivalent to the ``weaker" property
\begin{equation}
\label{eq: 16}
\cR_\phi(Q)[\xi;\omega]\geqsim_{p,\phi}\ |\omega_1||\omega_2|\,.
\end{equation}
We can relax the condition $Q\in C^2(\R^4)$ by requiring that $Q$ be of class $C^1$ and almost everywhere twice differentiable with second-order weak derivatives. Then we can consider the flow corresponding to a regularization of $Q$ by standard mollifiers (see Corollaries~\ref{t: 4} and \ref{l: sycn final stima}).

\bigskip
When $\cA$ is merely a generator of a symmetric contraction semigroup but not a differential-type operator, it is not so obvious that \eqref{eq: 16} implies \eqref{eq: heatflow p1 bis}. However, this is actually the case and will be proved in Section~\ref{s: proof bilinear}.

When $\phi=0$, similar heat-flow techniques corresponding to Bellman functions have so far been employed in the Euclidean case \cite{PV, NV, DV, DV-Sch, DV-Kato} and recently also in the Riemannian case \cite{CD}. For a different perspective on heat-flow techniques, various examples and references we refer the reader to the papers by Bennett et al. \cite{B, BCCT}.

\subsection*{Summary}
The proof of Theorem~\ref{t: sycn bil embedding} reduces to finding for every $p>2$ and $\phi\in(-\phi_p,\phi_p)$ a function $Q\in C^1(\R^4)$ such that \label{summary}
\begin{itemize}
\item[({\rm i})]
$0\leqslant Q(\zeta,\eta)\leqsim_{p,\phi} |\zeta|^p+|\eta|^q,\hskip 30pt \forall\zeta,\eta\in \R^2;$
\item[({\rm ii})]
$\mod{(\partial_{\zeta}Q)(\zeta,\eta)}\leqsim\,|\zeta|^{p-1}+|\eta|$
\hskip 7pt and \hskip 7pt
$\mod{(\partial_{\eta}Q)(\zeta,\eta)}\leqsim\,|\eta|^{q-1}+|\zeta|,\hskip 30pt \forall\zeta,\eta\in \R^2$;
\item[({\rm iii})] the function $Q$ is almost everywhere twice differentiable with second-order weak derivatives and
$$
\cR_\phi(Q)[\xi;\omega]\geqsim_{p,\phi}\ |\omega_1||\omega_2|\,,
$$
for almost every $\xi\in\R^4$ and all $\omega=(\omega_1,\omega_2)\in\R^2\times\R^2$.
\end{itemize}

It is not an easy matter to find functions $Q$ satisfying (i), (ii) and (iii) with ``optimal'' constants, even when $\phi=0$. It turned out \cite{NT,DV-Sch} that, if $\phi=0$, we can take for $Q$ the so-called Nazarov--Treil Bellman function which will be defined in the next section.
We show in this paper that the very same Nazarov--Treil Bellman function satisfies the condition (iii) in the {\sl whole} range of $\phi$ we are interested in, i.e., $(-\phi_p,\phi_p)$; see Theorem~\ref{t: convex} for the precise statement.

\section{Nazarov--Treil Bellman function}
\label{s: bellman}
In this section we fix $p\in (2,\infty)$ and $\epsilon\in (0,1/2)$. Set
\begin{equation}
\label{eq: def delta}
\delta=\delta(p,\e)=\frac{2q(q-1)\epsilon}{85}
\end{equation}
and define the Nazarov--Treil Bellman function $Q: \R^2\times \R^2\longrightarrow \R_+$ by
\begin{equation}
\label{eq: def bellman}
Q(\zeta,\eta)=
|\zeta|^{p}+|\eta|^{q}+\delta
\left\{
\aligned
& |\zeta|^2|\eta|^{2-q} & ; & \ \ |\zeta|^p\leq |\eta|^q\\
& \frac{2}{p}\,|\zeta|^{p}+\left(\frac{2}{q}-1\right)|\eta|^{q}
& ; &\ \ |\zeta|^p\geq |\eta|^q\,.
\endaligned\right.
\end{equation}
The reasons for the particular choice of $\delta$ in \eqref{eq: def delta} will only become apparent later; see Remark~\ref{r: reasons for delta}.

The origins of $Q$ lie in the paper of F. Nazarov and S. Treil \cite{NT}.
A modification of their function was later applied by A. Volberg and the second author in \cite{DV,DV-Sch}. Here we use a simplified variant which comprises only two variables. It was introduced in \cite{DV-Kato} and used by the present authors in \cite{CD}. The function $Q$ above is essentially the same as in \cite{DV-Kato,CD}, except for the fact that here $\delta$ also depends on a new parameter -- $\e$.

The construction of the original Nazarov--Treil function in \cite{NT} was one of the earliest examples of the so-called Bellman function technique, which was introduced in harmonic analysis shortly beforehand by Nazarov, Treil and Volberg \cite{NTV}.  Earlier, the Bellman functions have implicitly appeared in the work of Burkholder \cite{Bu1}, see also \cite{Bu2,Bu4}. If interested in the genesis of Bellman functions and the overview of the method, the reader is also referred to Volberg et al. \cite{NTV1,V,NT} and Wittwer \cite{W}.
The method has seen a whole series of applications, mostly in Euclidean harmonic analysis.

In the course of the last few years, the Nazarov--Treil function considered here was found to possess nontrivial properties that reach much beyond the need for which it was originally constructed in \cite{NT}.
These properties are used for proving several variants of the bilinear embedding. For example, such was the case in \cite{DV-Sch, DV-Kato} with general Schr\"odinger  operators.
As already mentioned, in the present paper we shall prove another convexity property of $Q$ (see Theorem~\ref{t: convex}).

\medskip
The following result is a direct consequence of the definition of $Q$.
\begin{proposition}\label{p: 3}
The function $Q$ belongs to $C^1(\R^4)$ and is of order $C^2$ everywhere {\it except} on the set
$
\Upsilon_0=\Mn{(\zeta,\eta)\in  \R^2\times\R^2}{(\eta=0)\vee (|\zeta|^p=|\eta|^q)}\,.
$

For every $(\zeta,\eta)\in\R^2\times\R^2$,
$$
0\leqslant Q(\zeta,\eta)\leqslant (1+\delta)\left(|\zeta|^p+|\eta|^q\right),
$$
and
\begin{align*}
&2\mod{(\partial_{\zeta}Q)(\zeta,\eta)}\leqslant (p+2\delta)\max\{|\zeta|^{p-1},|\eta|\},\\
&2\mod{(\partial_{\eta}Q)(\zeta,\eta)}\leqslant (q+(2-q)\delta)|\eta|^{q-1}.
\end{align*}
\end{proposition}
Next result is one of the key estimates of the paper and will be proved in Section~\ref{s: convexity}. Recall that $p_\e,q_\e$ and $\cR_\phi$ were defined in \eqref{eq: 15} and \eqref{eq: def of R_phi}, respectively.
In particular, by \eqref{d: 2},
$$
\cR_{\phi}(Q)[\xi;\omega]
=\Sk{\cH(Q)(\xi)\omega}{\cU_{\phi}\omega}_{\R^4}\,.
$$
\begin{theorem}
\label{t: convex}
For every $\phi\in[-\phi_{p_\epsilon},\phi_{p_\epsilon}]$,
\begin{equation*}
\cR_{\phi}(Q)[\xi;\omega]\,
\geqslant\,
2\delta \cos\phi
\,|\omega_1||\omega_2|\,,
\end{equation*}
for all $\xi\in(\R^2\times\R^2)\setminus\Upsilon_0$ and
$\omega=(\omega_1,\omega_2)\in\R^2\times\R^2$.
\end{theorem}
\begin{remark}
The special case $\phi=0$ follows from earlier papers, see \cite[Section 8]{NT} and \cite[Theorem 3]{DV-Sch}. Actually, the estimate for $\phi=0$ was essentially the very purpose of constructing the prototype of $Q$ in \cite{NT}.
\end{remark}

\subsection*{Regularization of $Q$}
Denote by $*$ convolution in $\R^{4}$, and let $(\psi_r)_{r>0}$ be a nonnegative, smooth and compactly supported approximation to the identity in $\R^4$. Since $Q\in C^1(\R^4)$ and its second-order partial derivatives exist on $\R^4\setminus \Upsilon_0$ and are locally integrable in $\R^4$,
\begin{equation}\label{eq: 8}
\cR_\phi(Q*\psi_r)[\xi;\omega]=\int_{\R^4}\cR_\phi(Q)[\xi-\xi';\omega]\psi_r(\xi')\wrt \xi'
\end{equation}
for every $r>0$, $\xi\in\R^4$ and every $\omega=(\omega_1,\omega_2)\in\R^2\times\R^2$.

\begin{corollary}
\label{t: 4}
For every $r>0$ and $\phi\in [-\phi_{p_\epsilon},\phi_{p_\epsilon}]$,
\begin{equation}\label{eq: sim}
\cR_{\phi}(Q*\psi_r)[\xi;\omega]\,
\geqslant\,
2\delta \cos\phi |\omega_1||\omega_2|\,,
\end{equation}
for all $\xi\in\R^4$ and every $\omega=(\omega_1,\omega_2)\in\R^2\times\R^2$.
\end{corollary}
\begin{proof}
The corollary immediately follows from \eqref{eq: 8} and Theorem~\ref{t: convex}.
\end{proof}

For use later in Section~\ref{s: proof bilinear} we state the main corollary of Theorem~\ref{t: convex}. If $\alpha,\beta\in \C$ and $\phi\in(-\pi/2,\pi/2)$, define $F_{\alpha,\beta}:\C^2\rightarrow \R$ by the rule
\begin{align}
\label{eq: final stima}
F_{\alpha,\beta}(\xi)=2\Re\left[e^{i\phi}\alpha\cdot(\partial_\zeta Q)(\xi)+e^{-i\phi}\beta\cdot(\partial_\eta Q)(\xi)\right]\,.
\end{align}
\begin{corollary}
\label{l: sycn final stima}
Let $\alpha,\beta\in\C$. Then, for every $\phi\in[-\phi_{p_\epsilon},\phi_{p_\epsilon}]$, and all $\zeta,\eta\in\C$,
\begin{align*}
F_{\alpha,\beta}(\alpha+\zeta,\beta+\eta)-F_{\alpha,\beta}(\zeta,\eta)\geqslant 2\delta\cos\phi |\alpha||\beta|\,.
\end{align*}
\end{corollary}
\begin{proof}
For every $r>0$, let $F_{r,\alpha,\beta}$ be the function defined by replacing $Q$ with $Q*\psi_r$ in \eqref{eq: final stima}. Since $Q$ is $C^1$, one has $\lim_{r\rightarrow 0^+}F_{r,\alpha,\beta}(\xi)=F_{\alpha,\beta}(\xi)$ for all $\xi\in\C^2$. Therefore, it suffices to prove the lemma with $F_{\alpha,\beta}$ replaced by $F_{r,\alpha,\beta}$.

If $u=(u_1,u_2)$ is a vector in $\C^2$, let us denote by $\cV(u)$ its real counterpart in $\R^{4}$, that is,
$$
\cV(u)=(\Re u_1,\Im u_1,\Re u_2, \Im u_2)\,.
$$
Whit this notation, we have
$$
F_{r,\alpha,\beta}(\xi)
=\Sk{\cU_\phi\cdot\cV(\alpha,\beta)}{\nabla (Q*\psi_r)(\xi)}_{\R^4}\,.
$$
Therefore, for all $\xi\in\C^2$,
$$
\Sk{\cV(\alpha,\beta)}{\nabla F_{r,\alpha,\beta}(\xi)}_{\R^4}= \cR_{\phi}(Q*\psi_r)[\xi;\cV(\alpha,\beta)].
$$
The result now follows from the mean value theorem and Corollary~\ref{t: 4}.
\end{proof}

\section{Proof of Theorem~\ref{t: convex}}
\label{s: convexity}
We shall throughout this section keep in mind that $p>2$ and $\epsilon \in (0,1/2]$ are fixed numbers, that $q=p/(p-1)$ is the conjugate exponent of $p$ and that $p_\e,q_\e$ were defined in \eqref{eq: 15}. Furthermore, recall that with $p,\e$ as above we introduced $\delta=\delta(p,\e)$ in \eqref{eq: def delta} and that the function $Q=Q_{p,\e}$ was defined in \eqref{eq: def bellman} by means of these $p,\delta$. Finally, bear in mind that $\cH(Q)$ denotes the Hessian matrix of $Q$, and that, for $\phi\in (-\pi/2,\pi/2)$, the matrix function $\cR_\phi(Q)$ was defined in \eqref{eq: def of R_phi}. We can write
\begin{equation*}
\label{eq: split R}
\cR_{\phi}(Q) =\cos\phi\cdot \cH(Q)+\sin\phi\cdot \cI(Q)\,
\end{equation*}
with the matrices $\cH(Q)$ and $\cI(Q)$ expressed in the block matrix notation as
\begin{equation*}
\label{eq: block}
\cH(Q)=\left[ {\begin{array}{lc}
\cH_1(Q)&\cH_2(Q)\\
\cH_2^T(Q)&\cH_3(Q)
\end{array} } \right]\, \ \ \ {\rm and} \ \ \  \cI(Q)=
\left[ {\begin{array}{lr}
\cI_1(Q)& \cI_2(Q)\\
\cI_2^T(Q)&-\cI_3(Q)
\end{array} } \right],
\end{equation*}
where, in turn, $\cI_j(Q)$ are $2\times 2$ real matrices given by
\begin{equation*}
\cI_1(Q)=
\left[
{\begin{array}{cc}
\partial^2_{\zeta_1\zeta_2}Q & \frac12(\partial^2_{\zeta_2\zeta_2}Q-\partial^2_{\zeta_1\zeta_1}Q)\\
\frac12(\partial^2_{\zeta_2\zeta_2}Q-\partial^2_{\zeta_1\zeta_1}Q) &-\partial^2_{\zeta_1\zeta_2}Q
\end{array} }
\right]\,,
\end{equation*}
\begin{equation*}
\cI_3(Q)=
\left[
{\begin{array}{cc}
\partial^2_{\eta_1\eta_2}Q & \frac12(\partial^2_{\eta_2\eta_2}Q-\partial^2_{\eta_1\eta_1}Q)\\
\frac12(\partial^2_{\eta_2\eta_2}Q-\partial^2_{\eta_1\eta_1}Q) &-\partial^2_{\eta_1\eta_2}Q
\end{array} }
\right]\,
\end{equation*}
and
\begin{equation*}
\cI_2(Q)=
\frac12\left[ {\begin{array}{rr}
\partial^2_{\zeta_2\eta_1}Q-\partial^2_{\zeta_1\eta_2}Q & \partial^2_{\zeta_1\eta_1}Q+\partial^2_{\zeta_2\eta_2}Q\\
-\partial^2_{\zeta_1\eta_1}Q-\partial^2_{\zeta_2\eta_2}Q& \partial^2_{\zeta_2\eta_1}Q-\partial^2_{\zeta_1\eta_2}Q
\end{array} } \right]\,.
\end{equation*}
\indent
Proving Theorem~\ref{t: convex} requires several steps. Given $\xi\in(\R^2\times\R^2)\setminus\Upsilon_0$ and $\omega=(\omega_1,\omega_2)\in\R^2\times\R^2$, we have
\begin{equation}
\label{eq: split in blocks}
\aligned
(\cos\phi)^{-1}
\cR_\phi(Q)\left[\xi;\omega\right]
& = \hskip 9pt \Sk{\big(\cH_1(Q)+\tan\phi\cdot \cI_1(Q)\big)(\xi)\omega_1}{\omega_1}_{\R^2} \\
&  \hskip 4pt + 2\Sk{\big(\cH_2(Q)+\tan\phi\cdot\cI_2(Q)\big)(\xi)\omega_1}{\omega_2}_{\R^2} \\
&  \hskip 6pt + \hskip 6pt\Sk{\big(\cH_3(Q)-\tan\phi\cdot\cI_3(Q)\big)(\xi)\omega_2}{\omega_2}_{\R^2} \,.
\endaligned
\end{equation}
We shall split the region $(\R^2\times\R^2)\backslash\Upsilon_0$ into a ``good" and a ``bad" part as
$$
\Omega_g=\Mn{(\zeta,\eta)\in\R^2\times\R^2}{|\zeta|^p>|\eta|^q>0} \hskip 11pt \text{and}\hskip 11pt  \Omega_b=\Mn{(\zeta,\eta)\in\R^2\times\R^2}{|\zeta|^p<|\eta|^q}.
$$
In both regions above we estimate the three terms on the right-hand side of \eqref{eq: split in blocks} one by one. The obtained results will be summarized in Propositions~\ref{p: 1} and \ref{p: 2}, respectively, and Theorem~\ref{t: convex} will then follow by combining them.

\subsection*{Power functions} Define, for $v\in\R^2$ and $r>1$,
$$
F_r(v)=|v|^r\,.
$$
The Nazarov--Treil function $Q$ essentially comprises three tensor products of power functions, namely
$F_p\otimes{\mathbf 1}$, $ {\mathbf 1}\otimes F_q$ and $F_2\otimes F_{2-q}$.
Our goal is  to study the quadratic form on $\R^4$ generated by the matrix $\cH(Q)\cU_\phi$,
with $\cU_\phi$ as in \eqref{planinic}.
We have, for example, $\cH(F_p\otimes{\mathbf 1})\cU_\phi=\big[\cH(F_p)\cO_\phi\big]\oplus{\mathbf 0}$.
Hence it is useful to understand first the properties of the quadratic forms on $\R^2$ induced by the matrices $\cH(F_r)\cO_\phi$, $r>1$. They have been earlier studied by Bakry \cite{Bakry3}. Define
$$
S_{r,\phi}(v)=
\frac12\,\Big(\cH(F_r)(v)\cO_\phi+\cO_{-\phi}\cH(F_r)(v)
\Big)\,.
$$
A calculation shows that
$$
S_{r,\phi}(v)=r|v|^{r-2}\cos\phi\cdot\cD_{r,\phi}(v)\,,
$$
where
\begin{equation*}
\cD_{r,\phi}(v)=\frac r2\Big(I_2+\frac{1-2/r}{\cos\phi}\,\cK(2\f-\phi)\Big)
\end{equation*}
with
$$
\cK(\alpha)
=\left[ {\begin{array}{rr}
  1&0\\
  0&-1
 \end{array} } \right]
 \cO_{-\alpha}
=\left[ {\begin{array}{rr}
  \cos\alpha&\sin\alpha\\
  \sin\alpha&-\cos\alpha
 \end{array} } \right]\,,
$$
and $\f$ being the polar angle of $v$.

\begin{lemma}[\cite{Bakry3}]
\label{l: bakry}
Let $\phi\in(-\pi/2,\pi/2)$ and $r>1$. For every $v\in\R^2\backslash\{0\}$,
the matrix $\cD_{r,\phi}(v)$ induces a positive semidefinite quadratic form on $\R^2$ if and only if $|\phi|\leqslant \phi_r$.
\end{lemma}

\begin{proof}
Note that ${\rm tr}\,\cD_{r,\phi}(v)=r>0$ and ${\displaystyle \det\cD_{r,\phi}(v)=(r/2)^2[1-(\cos\phi_r/\cos\phi)^2]}$, which is nonnegative if and only if $|\phi|\leqslant \phi_r$.
\end{proof}

In the continuation we use Lemma~\ref{l: bakry} for estimating the terms in \eqref{eq: split in blocks}. We also need the following identities which can be easily deduced from the very definition of $\cD_{r,\phi}$.
\begin{align}
&\cD_{p,\phi}(v)=\epsilon I_2+(1-\epsilon)\cD_{p_\epsilon,\phi}(v)\,,\label{eq: D1}\\
&\cD_{q,\phi}(v)=\epsilon(q-1)I_2+[1-\epsilon(q-1)]\cD_{q_\epsilon,\phi}(v)\,,\label{eq: D2}\\
&(2-q)\cD_{2-q,\phi}(v)=-2(q-1)I_2+q\cD_{q,\phi}(v)\,.\label{eq: D3}
\end{align}

\begin{proposition}\label{p: 1}
For all $\xi\in\Omega_g$, $\omega=(\omega_1,\omega_2)\in\R^2\times\R^2$ and $\phi\in[-\phi_{p_\e},\phi_{p_\e}]$ we have
$$
\cR_{\phi}(Q)[\xi;\omega]
\geqslant
2\delta\cos\phi
|\omega_1||\omega_2|
\,.
$$
\end{proposition}

\begin{proof}
First note that in the region $\Omega_g$ we have $\cH_2(Q)=\cI_2(Q)=0$. As for the other blocks on the right-hand side of \eqref{eq: split in blocks},
$$
\aligned
\left(\cH_1(Q)+\tan\phi\cdot \cI_1(Q)\right)(\zeta,\eta)
&=(p+2\delta)|\zeta|^{p-2}\cD_{p,\phi}(\zeta)\\
\left(\cH_3(Q)-\tan\phi\cdot\cI_3(Q)\right)(\zeta,\eta)
&=(q+(2-q)\delta)|\eta|^{q-2}\cD_{q,-\phi}(\eta)\,.
\endaligned
$$
By Lemma~\ref{l: bakry}, the quadratic forms $\cD_{p_\e,\phi}(\zeta)$ and $\cD_{q_\e,-\phi}(\eta)$ are positive semidefinite, because  $|\phi|\leq \phi_{p_\e}=\phi_{q_\e}$ by assumption. Therefore, by \eqref{eq: D1} and \eqref{eq: D2},
\begin{align*}
\left(\cH_1(Q)+\tan\phi\cdot \cI_1(Q)\right)(\zeta,\eta)
& \geqslant \epsilon p|\zeta|^{p-2}I_2\\
 \left(\cH_3(Q)-\tan\phi\cdot\cI_3(Q)\right)(\zeta,\eta)
& \geqslant \epsilon(q-1) q|\eta|^{q-2}I_2
\,.
\end{align*}
It follows from \eqref{eq: split in blocks} that
$$
\cR_{\phi}(Q)[\xi;\omega]
\geqslant
\cos\phi\big(\epsilon p|\zeta|^{p-2}|\omega_1|^2+\epsilon q(q-1)|\eta|^{q-2}|\omega_2|^2\big)\,.
$$
Since in $\Omega_g$ one has that $|\eta|^{q-2}> |\zeta|^{2-p}$, we can at this point quickly finish
the proof.
\end{proof}
We now estimate the terms on the right-hand side of \eqref{eq: split in blocks} in the bad region $\Omega_b$. Define
$$
G(\zeta,\eta)=|\zeta|^p+|\eta|^q
\hskip 25pt
\text{and}
\hskip 25pt
H(\zeta,\eta)=|\zeta|^2|\eta|^{2-q}\,,
$$
so that in $\Omega_b$ we have
\begin{equation}
\label{eq: 18}
Q=G+\delta H\,.
\end{equation}
\begin{lemma}\label{l: 3}
For every $\xi=(\zeta,\eta)\in\Omega_b$, $\omega_1\in\R^2$ and $\phi\in[-\phi_p, \phi_p]$ one has
$$
\Sk{\big(\cH_1(Q)+\tan\phi\cdot\cI_1(Q)\big)(\xi)\omega_1}{\omega_1}_{\R^2}\geqslant 2\delta|\eta|^{2-q}|\omega_1|^2\,.
$$
\end{lemma}
\begin{proof}
We have
$$
\big(\cH_1(G)+\tan\phi\cdot\cI_1(G)\big)(\zeta,\eta)=p|\zeta|^{p-2}\cD_{p,\phi}(\zeta)
$$
as well as
$$
\cH_1(H)(\zeta,\eta)=2|\eta|^{2-q}I_2
\hskip 25pt
\text{and}
\hskip 25pt
\cI_1(H)(\zeta,\eta)=0\,.
$$
The required estimate now follows from \eqref{eq: 18} and Lemma~\ref{l: bakry}.
\end{proof}

\begin{lemma}
\label{l: 2}
For every $\xi \in\Omega_b$, $\omega_1,\omega_2\in\R^2$ and  $\phi\in[-\phi_{p},\phi_{p}]$ one has
\begin{equation*}
\Sk{\big(\cH_2(Q)+\tan\phi\cdot\cI_2(Q)\big)(\xi)\omega_1}{\omega_2}_{\R^2}
\geqslant
-8\delta|\omega_1||\omega_2|\,.
\end{equation*}
\end{lemma}
\begin{proof}
Observe that $\cH_2(Q)+\tan\phi\cdot\cI_2(Q)=\delta(\cH_2(H)+\tan\phi\cdot \cI_2(H))$. Since
$$
\big(\partial^2_{\zeta_i\eta_j}H\big)(\zeta,\eta)=2(2-q)\zeta_i\eta_j|\eta|^{-q},
$$
the absolute values of the entries of $\cH_2(H)$ and $\cI_2(H)$ are uniformly bounded in $\Omega_b$ by $2(2-q)$.
Consequently,
$$
\Sk{\big(\cH_2(Q)+\tan\phi\cdot\cI_2(Q)\big)(\xi)\omega_1}{\omega_2}_{\R^2}
\geqslant -4\delta(2-q)(1+|\tan\phi|)|\omega_1||\omega_2|\,.
$$
We can now use the assumption on $\phi$ and estimate
$$
(2-q)(1+|\tan\phi|)
\leqslant
(2-q)(1+\tan \phi_{p}) =
2-q+\frac{2}{\sqrt{p-1}} <
2\,.
$$
This finishes the proof of Lemma~\ref{l: 2}.
\end{proof}

\begin{lemma}\label{l: 4}
For every $\xi=(\zeta,\eta)\in \Omega_b$, $\omega_2\in\R^2$ and $\phi\in[-\phi_{p_\e},\phi_{p_\e}]$,
$$
\Sk{\big(\cH_3(Q)-\tan\phi\cdot\cI_3(Q)\big)(\xi)\omega_2}{\omega_2}_{\R^2}\geqslant \frac{81}{2}\delta|\eta|^{q-2}|\omega_2|^2\,.
$$
\end{lemma}

\begin{proof}
We start by writing
\begin{align}
(\cH_3(G)-\tan\phi\cdot\cI_3(G))(\zeta,\eta)&=q|\eta|^{q-2}\cD_{q,-\phi}(\eta)\,,\label{eq: 5}\\
\left(\cH_3(H)-\tan\phi\cdot\cI_3(H)\right)(\zeta,\eta)&=(2-q)|\zeta|^2|\eta|^{-q} \cD_{2-q,-\phi}(\eta)\,.\nonumber
\end{align}
Since the determinant of the matrix $\cD_{2-q,-\phi}(\eta)$ is negative, we consider two cases.

\medskip
{\bf Case 1:} Suppose that $\sk{\cD_{2-q,-\phi}(\eta)\omega_2}{\omega_2}_{\R^2}\geqslant 0$. Then from \eqref{eq: 18}, \eqref{eq: 5} and \eqref{eq: D2} we get
\begin{align*}
\sk{(\cH_3(Q)&-\tan\phi\cdot\cI_3(Q))(\zeta,\eta)\omega_2}{\omega_2}_{\R^2}\geqslant \sk{(\cH_3(G)-\tan\phi\cdot\cI_3(G))(\zeta,\eta)\omega_2}{\omega_2}_{\R^2}\\
&= q(q-1)\epsilon|\eta|^{q-2}|\omega_2|^2+q[1-\epsilon(q-1)]|\eta|^{q-2}\Sk{\cD_{q_\e,-\phi}(\eta)\omega_2}{\omega_2}_{\R^2}\,.
\end{align*}
The required estimate now follows since the last term in the right-hand side is nonnegative by Lemma~\ref{l: bakry},
because $|\phi|\leqslant\phi_{p_\e}=\phi_{q_\e}$.

\medskip
{\bf Case 2:} Suppose that $\Sk{\cD_{2-q,-\phi}(\eta)\omega_2}{\omega_2}_{\R^2}\leqslant 0$.
Since in $\Omega_b$ we have $|\zeta|^2<|\eta|^{2q-2}$, identities \eqref{eq: 18}, \eqref{eq: 5} imply
$$
\aligned
\Sk{(\cH_3(Q)-\tan\phi\cdot\cI_3 (Q))(\zeta,\eta)\omega_2}{\omega_2}_{\R^2} &\\
& \hskip -100pt \geq|\eta|^{q-2}\left(q\Sk{\cD_{q,-\phi}(\eta)\omega_2}{\omega_2}_{\R^2}+\delta(2-q)\Sk{\cD_{2-q,-\phi}(\eta)\omega_2}{\omega_2}_{\R^2}\right) \,.
\endaligned
$$
From \eqref{eq: D3} we get
$$
q\cD_{q,-\phi}(\eta)+\delta(2-q)\cD_{2-q,-\phi}(\eta)=q(\delta+1)\cD_{q,-\phi}(\eta)-2\delta(q-1)I_2\,.
$$
By Lemma~\ref{l: bakry} the quadratic form $\cD_{q_\epsilon,\-\phi}(\eta)$ is nonnegative, because $|\phi|\leqslant\phi_{p_\e}=\phi_{q_\e}$. Therefore, it follows from \eqref{eq: D2} and from the definition of $\delta$ \eqref{eq: def delta} that
\begin{align*}
\Sk{(\cH_3(Q)-\tan\phi\cdot\cI_3 (Q))(\zeta,\eta)\omega_2}{\omega_2}_{\R^2}
 &\geq \left(q(q-1)\epsilon(\delta+1)-2\delta(q-1)\right)|\eta|^{q-2}|\omega_2|^2\\
 &\geq \frac{81}{2}\delta|\eta|^{q-2}|\omega_2|^2\,,
\end{align*}
which concludes the proof of the lemma.
\end{proof}

\begin{proposition}
\label{p: 2}
For all $\xi\in\Omega_b$, $\omega=(\omega_1,\omega_2)\in\R^2\times\R^2$ and $\phi\in[-\phi_{p_\e},\phi_{p_\e}]$ we have
$$
\cR_{\phi}(Q)[\xi;\omega]
\geqslant\, 2\delta\cos\phi\,|\omega_1||\omega_2|
\,.
$$
\end{proposition}
\begin{proof}
Combine \eqref{eq: split in blocks} with Lemmas~\ref{l: 3}, \ref{l: 2} and \ref{l: 4}.
\end{proof}

\begin{remark}
\label{r: reasons for delta}
Let us explain why $\delta$ was chosen in \eqref{eq: def delta} the way it was.

By taking for $\delta$ a number of the form $Dq(q-1)\e$ for some $\e>0$ to be determined, the proofs of Lemmas
\ref{l: 3} and \ref{l: 2} would not change at all, while in Lemma \ref{l: 4} we would get
$
\langle\,\cdot\,\rangle\geq(1/D-2)\delta|\eta|^{q-2}|\omega_2|^2\,,
$
and consequently in Proposition \ref{p: 2} we would get
$$
\cR_{\phi}(Q)[\xi;\omega]
\geqslant\, (\underbrace{2\sqrt{2/D-4}-16}_{(*)})
\delta\cos\phi\,|\omega_1||\omega_2|
\,.
$$
If we want $(*)$ to equal 2, we must choose $D=2/85$.
\end{remark}

\section{Proof of Theorem~\ref{t: sycn bil embedding}}\label{s: proof bilinear}
Fix $p\in(2,\infty)$, $\epsilon\in (0,1/2)$ and define $\delta$ and $Q$ as in \eqref{eq: def delta} and \eqref{eq: def bellman}. In order to prove the bilinear embedding we apply the heat-flow method that we described in Section~\ref{s: heat-flow}.

The pointwise estimates of $Q$ and its partial derivatives (see Proposition~\ref{p: 3}) in particular imply that $Q$ satisfies \eqref{es: 1} and \eqref{es: 2}, so that the associated flow $\cE$ is regular on $L^p(\Omega,\nu)\times L^q(\Omega,\nu)$. Moreover, $\cE$ satisfies \eqref{eq: heatflow p2} with $A_0=1+\delta$, for all $\phi\in[-\phi_{p_\e},\phi_{p_\e}]$. Therefore, what is left to prove is that $\cE$ satisfies \eqref{eq: heatflow p1} or, equivalently, that $\cA$ satisfies \eqref{eq: heatflow p1 bis} for an appropriate $B_0>0$. In this section (see Corollary~\ref{c: sycn 1}) we shall prove that $\eqref{eq: heatflow p1 bis}$ holds with $B_0=2\delta\cos\phi$, for any $\phi\in[-\phi_{p_\e},\phi_{p_\e}]$; Theorem~\ref{t: sycn bil embedding} will follow immediately.\\

We first consider the special case when $\cA$ is the heat generator on the two-point space, i.e.,
$$
\cA=\cG=\left[ {\begin{array}{rr}
 1& -1\\
-1& 1
 \end{array} } \right]
$$
acting on $\C^2=L^\infty(\{a,b\}, \nu_{a,b})$, where  $\nu_{a,b}=(\delta_{a}+\delta_{b})/2$. It is easy to see that $\cG$ generates a symmetric contraction semigroup on $(\{a,b\},\nu_{a,b})$, e.g. \cite[Example 1]{KR}.
\begin{proposition}
\label{l: two-point}
For all $\phi\in [-\phi_{p_\epsilon},\phi_{p_\epsilon}]$ and every $f,g:\{a,b\}\rightarrow\C$,
\begin{equation*}
2\delta\cos\phi \mod{\int_{\{a,b\}}\cG f\cdot \overline{g}\wrt\nu_{a,b}}
\leq
2\Re\int_{\{a,b\}} \left(e^{i\phi}(\partial_\zeta Q)(f,g)\cG f+ e^{-i\phi}(\partial_\eta Q)(f,g)\cG g\right)\wrt\nu_{a,b}.
\end{equation*}
\end{proposition}
\begin{proof}
Since for every $u,v:\{a,b\}\rightarrow \C$ one has
$$
\int_{\{a,b\}}\cG u\cdot v\wrt\nu_{a,b}=\frac{1}{2}\left(u(a)-u(b)\right)\left(v(a)-v(b)\right),
$$
the proposition follows from Corollary~\ref{l: sycn final stima} applied with $\alpha=f(a)-f(b)$, $\beta=g(a)-g(b)$, $\zeta=f(b)$ and $\eta=g(b)$.
\end{proof}

\subsection*{The general case}
The idea of proving \eqref{eq: heatflow p1 bis} in the general case is to verify first a version of \eqref{eq: heatflow p1 bis} with $\cA$ replaced by $(I-T_t)/t$, for all $t>0$. This is the content of Proposition \ref{p: 4} below. In order to take advantage of the Bellman function's properties, we suitably decompose $I-T_t$ into a sum of two operators which we then tame separately.
At the end we pass to the limit as $t\rightarrow0$.

Before proceeding to Proposition \ref{p: 4} we need a ``representation formula'' for symmetric contractions (see Lemma~\ref{l: kernel rep.} below), which is a refined version of that used in \cite[Theorem 3.9]{O} for proving Proposition~\ref{t: 1} in the case of sub-Markovian semigroups. For the purpose of stating the representation formula we introduce a bit of notation and recall a few known results concerning the linear modulus of an operator and the Gelfand transform.

\begin{definition}
A bounded self-adjoint operator $T$ on $L^2(\Omega,\nu)$ is said to be a {\it symmetric contraction} on $(\Omega,\nu)$ if
for all $p\in [1,\infty]$ and $f\in L^2(\Omega,\nu)\cap L^p(\Omega,\nu)$,
\begin{equation*}
\|Tf\|_p\leq \norm{f}{p}\,.
\end{equation*}
We say $T$ is {\it sub-Markovian} if, in addition, $Tf\geq 0$ whenever $f\in L^2(\Omega,\nu)$ and $f\geq 0$. A sub-Markovian operator $T$ such that $T1=1$ is called {\it Markovian}.
\end{definition}

Following \cite{CK} and \cite{KRE}, denote by $\cP$ the family of all finite measurable partitions of $\Omega$, partially ordered in the usual way, i.e., $\pi\leq \pi'$ if and only if $\pi'$ is a refinement of $\pi$. If $T$ is a symmetric contraction on $(\Omega,\nu)$, $\pi=\{B_1,\dots,B_n\}$, and $f\in L^1(\Omega,\nu)$, set
$$
\mathbf{T}_\pi f=\sum^n_{j=1}|T(f\chi_{B_j})|.
$$
Next result was proven in \cite{CK} and \cite[Lemma 3.4]{Tag}, see also \cite[Theorems 4.1.2, 4.1.3]{KRE}.
\begin{lemma}
\label{l: sycn cont}
Suppose that $T$ is a symmetric contraction on $(\Omega,\nu)$. Then there exists a unique sub-Markovian operator $\mathbf{T}$ on $(\Omega,\nu)$ with the following properties:
\begin{enumerate}[{\rm(a)}]
\item
the operator norms of $T$ and $\mathbf{T}$ on $L^1(\Omega,\nu)$ are equal;
\item
\label{Streng}
$|Tf|\leq \mathbf{T}|f|$ whenever $f\in L^r(\Omega,\nu)$ and $1\leq r\leq \infty$.
\end{enumerate}
When $f\in L^1(\Omega,\nu)$ and $f\geq 0$, one has that
\begin{enumerate}[{\rm(a)}]
\addtocounter{enumi}{2}
\item
$\mathbf{T}f = \sup\Mn{|Tg|}{g \in L^1(\Omega,\nu), |g|\leq f}$;
\item
\label{eq: ggg}
$
\mathbf{T}f=\lim_{\pi\in\cP}\mathbf{T}_\pi f
$
in $L^1(\Omega,\nu)$.
\end{enumerate}
\end{lemma}
\noindent
The operator $\mathbf{T}$ is called {\it linear modulus} of $T$.

\subsection*{Gelfand theory}
Let us, for the reader's convenience, briefly summarize a few basic facts about the Gelfand theory we will need in the continuation. They are taken from \cite[Chapter 1]{TAKE} and \cite[Chapter 11]{RU2}.

Suppose that $\nu(\Omega)<\infty$. Denote by $\wh{\Omega}$ the maximal ideal space of the commutative unital $C^*$-algebra $L^\infty(\Omega,\nu)$. The Gelfand isomorphism $\cF: L^\infty(\Omega,\nu)\rightarrow C(\wh{\Omega})$ is given by $(\cF f)(x)=x(f)$. We have the identities
$\cF(fg)=\cF f\cdot \cF g$, $\overline{\cF f}=\cF\bar{f}$
and $\cF(|f|^r)=|\cF f|^r$ for all $r\in [1,\infty)$.
We will often write $\cF f=\hat f$ and $\cF^{-1}\f =\check{\f}$.
Since $\wh{\Omega}$ is a compact Hausdorff space, by the Riesz representation theorem, the measure $\nu$ is transported to a positive Radon measure $\hat{\nu}$ on $\wh{\Omega}$ such that
\begin{equation}
\label{eq: gelfand}
\int_{\Omega}f\wrt \nu=\int_{\wh{\Omega}}\hat f\wrt \hat{\nu}\,
\end{equation}
for all $f\in L^\infty(\Omega,\nu)$.
Moreover, every $f\in L^\infty(\wh{\Omega},\hat{\nu})$ has a representative in $C(\wh{\Omega})$, so that $L^\infty(\wh{\Omega},\hat{\nu})$ and $C(\wh{\Omega})$ coincide as Banach spaces. It follows that, if $T$ is a symmetric contraction on $(\Omega,\nu)$, with $\nu(\Omega)<\infty$, then $\cF T\cF^{-1}$ extends to a symmetric contraction on $(\wh{\Omega},\hat{\nu})$ which leaves $C(\wh{\Omega})$ invariant.

If $f$ and $g$ belong to $L^\infty(\Omega,\nu)$, then we denote by $\wh{f}\otimes \wh{g}$ the continuous function on $\wh{\Omega}\times\wh{\Omega}$ mapping  $(x,y)\mapsto \wh{f}(x)\wh{g}(y)$. We denote by $E$ the subspace of $C(\wh{\Omega}\times\wh{\Omega})$ comprising all real functions of the form $h=\sum_{j=1}^n\wh{f}_j\otimes {\wh{g}_j^{}}$. If $h\in E $ and $y\in\wh\Omega$, then we set $h_y=h(\cdot,y)$. We say that a complex measure $\mu$ on $\widehat\Omega\times\widehat\Omega$ is {\it symmetric} if $\wrt\mu(y,x)=\wrt\overline\mu(x,y)$.

\begin{lemma}
\label{l: kernel rep.}
Suppose that $\nu(\Omega)<\infty$. Then for every symmetric contraction $T$ on $(\Omega,\nu)$ there exists a complex symmetric Radon measure $m_T$ on $\wh{\Omega}\times\wh{\Omega}$ with the following properties:
\begin{itemize}
\item[{\rm (i)}]
for all $f,g\in L^\infty(\Omega,\nu)$,
$$
\sk{Tf}{g}
=\int_{\wh{\Omega}\times\wh{\Omega}}
\hat f\otimes\overline{\hat g}\wrt m_T
\,;
$$
\item[{\rm (ii)}]
if $\mathbf{T}$ denotes the linear modulus of $T$, then $m_{\mathbf{T}}$ coincides with the total variation of $m_T$. In particular, for all $f,g\in L^\infty(\Omega,\nu)$,
$$
\sk{\mathbf{T}f}{g}
=\int_{\wh{\Omega}\times\wh{\Omega}}
\hat f\otimes\overline{\hat g}\wrt |m_T|\,.
$$
\end{itemize}
\end{lemma}
\begin{proof}
Part {\rm (i)} can be proven by appropriately modifying the proof of \cite[Lemma 1.4.1]{F} (see also \cite[pp. 90--91]{O}). We omit the details.

Now let us turn towards (ii). By Lemma~\ref{l: sycn cont} \eqref{Streng} and the identity $|\cF \varphi|=\cF|\varphi|$, for all $\varphi\in C(\wh{\Omega})$ and every $x\in \wh{\Omega}$ one has that
\begin{equation}
\label{eq: 17}
\mod{(\cF T\cF^{-1}\varphi)(x)}
\leq
(\cF\mathbf{T}\cF^{-1} |\varphi|)(x)\,.
\end{equation}
Let $h=\sum_{j=1}^n\wh{f}_j\otimes {\wh{g}_j^{}}\in E$. Suppose first that $h\geq 0$. Then it follows from item (i), \eqref{eq: gelfand} and \eqref{eq: 17} that

\begin{align}
\label{eq: sycn 5}
\mod{\int_{\wh\Omega\times\wh\Omega}h\,\wrt m_T}
&=\bigg|\int_{\wh\Omega}\sum^n_{j=1}(\cF T\cF^{-1}\wh{f}_j)(x){\wh{g}_j}(x)\wrt\hat{\nu}(x)\bigg| \nonumber \\
&=\mod{\int_{\wh\Omega}(\cF T\cF^{-1}h_x)(x)\wrt\hat{\nu}(x)} \nonumber \\
&\leq \int_{\wh\Omega}\sum^n_{j=1}(\cF \mathbf{T}\cF^{-1}h_x)(x)\wrt\hat{\nu}(x)\\
&= \int_{\wh\Omega\times\wh\Omega}h\, \wrt m_{\mathbf{T}}\nonumber\,.
\end{align}

Suppose now that $h$ is real-valued, but not necessarily of constant sign.
Following \cite[p. 91]{O}, consider a sequence of polynomials $P_n$ such that $P_n(t)\geq 0$ for all $t\in [-1,1]$ and $P_n(t)\rightarrow |t|$ uniformly for $t\in[-1,1]$. In particular, for all $n\in\N$ we have $P_n(h/\|h\|_\infty)\geq 0$. Therefore, by \eqref{eq: sycn 5},
\begin{align*}
\frac{1}{\|h\|_\infty}\mod{\int_{\wh\Omega\times\wh\Omega}|h|\wrt m_T}&=\lim_{n\rightarrow\infty}\mod{\int_{\wh\Omega\times\wh\Omega}P_n(h/\|h\|_\infty)\wrt m_T}\\
&\leq \lim_{n\rightarrow\infty}\int_{\wh\Omega\times\wh\Omega}P_n(h/\|h\|_\infty)\wrt m_{\mathbf{T}}\\
&=\frac{1}{\|h\|_\infty}\int_{\wh\Omega\times\wh\Omega}|h|\wrt m_{\mathbf{T}}.
\end{align*}
Since $E$ is dense in $C(\wh\Omega\times\wh\Omega;\R)$, we obtain
\begin{align}\label{eq: sycn 4}
|m_T|\leq m_{\mathbf{T}}.
\end{align}

Fix $f,g\in L^\infty(\Omega,\nu)$, $f,g\geq 0$, and $\epsilon>0$. According to Lemma~\ref{l: sycn cont} \eqref{eq: ggg}, there exists a partition $\pi=\{B_1,\dots,B_n\}$ of $\Omega$ such that
\begin{equation*}
\langle\mathbf{T}f,g\rangle\leq \sum^n_{j=1}\int_\Omega|T(\chi_{B_j} f)|g\wrt\nu+\epsilon\,.
\end{equation*}
From item ${\rm (i)}$ and the fact that $\cF$ is an isomorphism of algebras we get
\begin{equation*}
\aligned
\int_\Omega\sum^n_{j=1}|T(\chi_{B_j}f)|g\wrt\nu
&=\int_\Omega \sum^n_{j=1}T(\chi_{B_j}f)\frac{\overline{T(\chi_{B_j}f)}}{|T(\chi_{B_j}f)|}g\wrt\nu\\
&=\int_{\wh\Omega\times\wh\Omega}\sum^n_{j=1}\cF(\chi_{B_j})(x)\cF(f)(x)\cF\left(
\frac{\overline{T(\chi_{B_j}f)}}{|T(\chi_{B_j}f)|}g\right)(y)\wrt m_T(x,y)\\
&\leq \int_{\wh\Omega\times\wh\Omega}\cF\left(\sum^n_{j=1}\chi_{B_j}\right)(x)\cF(f)(x)\cF(g)(y)\wrt|m_T|(x,y)\\
&=\int_{\wh\Omega\times\wh\Omega}
\hat f\otimes\hat g\wrt|m_T|.
\endaligned
\end{equation*}
It follows that for all nonnegative $f,g\in L^\infty(\Omega,\nu)$,
\begin{equation}\label{eq: 14}
\int_{\wh\Omega\times\wh\Omega}\hat f\otimes\hat g\wrt m_{\mathbf{T}}=\langle\mathbf{T}f, g\rangle  \leq \int_{\wh\Omega\times\wh\Omega}\hat f \otimes\hat g\, \wrt|m_T|\,.
\end{equation}
By combining \eqref{eq: sycn 4} with \eqref{eq: 14}, we obtain
$$
\int_{\wh\Omega\times\wh\Omega}h \wrt m_{\mathbf{T}} =\int_{\wh\Omega\times\wh\Omega}h \wrt|m_T|
$$
for all $h\in E$, and the density of $E$ in $C(\wh\Omega\times\wh\Omega;\R)$ implies item ${\rm (ii)}$.

The symmetry of $m_T$ follows from the operator $T$ being self-adjoint.
\end{proof}

\begin{proposition}
\label{p: 4}
Let $p\in (2,\infty)$ and $\epsilon\in (0,1/2)$. Define $\delta$ and $Q$ as in \eqref{eq: def delta} and \eqref{eq: def bellman} respectively.
Suppose that $T$ is a symmetric contraction on a $\sigma$-finite measure space $(\Omega,\nu)$. Then, for every $\phi\in [-\phi_{p_\epsilon},\phi_{p_\epsilon}]$, $f\in L^p(\Omega,\nu)$ and $g\in L^q(\Omega,\nu)$,
\begin{align*}
2\delta\cos\phi
&  \mod{\int_{\Omega}(I-T)(f) \overline{g}\wrt\nu} \\
& \leqslant 2\Re\int_\Omega \left[e^{i\phi}(\partial_\zeta Q)(f,g)(I-T)(f)+ e^{-i\phi}(\partial_\eta Q)(f,g)(I-T)(g)\right]\wrt\nu\,.
\end{align*}
\end{proposition}

\begin{proof}
First we show that it suffices to prove the proposition in the case when $f,g$ are bounded and supported in sets of finite measure. Indeed, suppose that $f\in L^p(\Omega,\nu)$ and $g\in L^q(\Omega,\nu)$ and consider sequences $(s_n)$ and $(t_n)$ of simple functions supported in sets of finite measure, converging a.e. to $f$ and $g$, respectively, and such that $|s_n|\leq |f|$, $|t_n|\leq |g|$.
By Proposition~\ref{p: 3}, $|(\partial_\zeta Q)(s_n,t_n)| \leqsim \max\{|f|^{p-1},|g|\}\in L^q(\Omega,\nu)$
and $|(\partial_\eta Q)(s_n,t_n)|\leqsim|g|^{q-1} \in L^p(\Omega,\nu)$.
Consequently, assuming that the proposition holds for all pairs $s_n,t_n$, the dominated convergence theorem and Lemma~\ref{l: sycn cont} \eqref{Streng} show that it then also holds for $f,g$.

As in \cite[proof of Theorem 3.9]{O}, we may further assume
that $\nu(\Omega)<\infty$. Indeed, if this is not the case, then we can replace $\Omega$ with $\Omega_0=({\rm supp}\,f)\cup({\rm supp}\,g)$
and $T$ with $\chi_{\Omega_0}T\chi_{\Omega_0}$, which is a symmetric contraction on $(\Omega_0,\nu)$.

Therefore, from now on we assume that $\nu(\Omega)<\infty$ and $f,g\in L^\infty(\Omega,\nu)$.\\

Recall that $\mathbf{T}$ denotes the linear modulus of $T$. We begin by splitting
\begin{align*}
2\Re\int_\Omega \left[e^{i\phi}(\partial_\zeta Q)(f,g)(I-T)(f)+ e^{-i\phi}(\partial_\eta Q)(f,g)(I-T)(g)\right]\wrt\nu
=I_1+I_2\,,
\end{align*}
where
\begin{align*}
&I_1=\int_\Omega(1-\mathbf{T}(1))2\Re\left[e^{i\phi}(\partial_\zeta Q)(f,g)f+ e^{-i\phi}(\partial_\eta Q)(f,g)g\right]\wrt\nu,  \\
&I_2=\int_\Omega  2\Re\left[e^{i\phi}(\partial_\zeta Q)(f,g)(\mathbf{T}(1)I-T)f+ e^{-i\phi}(\partial_\eta Q)(f,g)(\mathbf{T}(1)I-T)g\right]\wrt\nu\,.
\end{align*}
We first estimate $I_1$. An easy computation shows that
$$
2(\partial_\zeta Q)(f,g)f\geqslant p|f|^p\hskip 30pt  {\rm and}\hskip 30pt   2(\partial_\eta Q)(f,g)g\geqslant q|g|^q.
$$
Since $1-\mathbf{T}(1)\geqslant 0$
and $p|f|^p+q|g|^q\geqslant p^{1/p}q^{1/q}|f||g|\geqslant 2\delta |f||g|$, it follows that
\begin{align}
\label{eq: sycn final 1}
I_1\geqslant \int_\Omega(1-\mathbf{T}(1))\cos\phi(p|f|^p+q|g|^q)\wrt\nu \geq 2\delta\cos\phi \mod{\int_\Omega(1-\mathbf{T}(1))f\overline{g}\wrt\nu}.
\end{align}
We now estimate $I_2$. Write
\begin{equation}
\label{eq: sycn final 3}
I_2=\Re\left(e^{i\phi}K_1+e^{-i\phi}K_2\right)\,,
\end{equation}
where
\begin{equation*}
K_1=2\int_\Omega  \left[\mathbf{T}(1)f-Tf\right](\partial_\zeta Q)(f,g)\wrt\nu
\hskip 10pt
{\rm and}
\hskip 10pt
K_2=2\int_\Omega  \left[\mathbf{T}(1)g-Tg\right](\partial_\eta Q)(f,g)\wrt\nu\,.
\end{equation*}
Let $m_T$ and $m_\mathbf{T}$ be the two measures given by  Lemma~\ref{l: kernel rep.}. By the polar decomposition, there exists a measurable function $k:\wh\Omega\times \wh\Omega\rightarrow \C$ such that $|k(x,y)|=1$ for all $x,y\in\wh\Omega$ and
$\wrt m_{T} =k \wrt |m_T |=k \wrt m_{\mathbf{T}} \,.$
From the definition of $Q$ and the properties of the Gelfand isomorphism $\cF$, it follows that
$\cF[(\partial_\zeta Q)(f,g)]=(\partial_\zeta Q)(\hat f,\hat g)$. Hence Lemma~\ref{l: kernel rep.} implies that
$$
K_1=
2\int_{\wh\Omega\times\wh\Omega}
\left(\hat f(y)-\hat f(x)k(x,y)\right)\,
(\partial_\zeta Q)\left(\hat f(y),\hat g(y)\right)
\wrt |m_T|(x,y)\,.
$$
Since $m_T$ is symmetric, $|m_T|$ is symmetric and $k(y,x)=\overline{k(x,y)}$. Consequently,
$$
\aligned
K_1=&\hskip12,3pt\int_{\wh\Omega\times\wh\Omega}
\left(\hat f(y)-\hat f(x)k(x,y)\right)\,
(\partial_\zeta Q)\left(\hat f(y),\hat g(y)\right)
\wrt |m_T|(x,y)\\
&+\int_{\wh\Omega\times\wh\Omega}
\left(\hat f(x)-\hat f(y)\overline{k(x,y)}\right)\,
(\partial_\zeta Q)\left(\hat f(x),\hat g(x)\right)
\wrt |m_T|(x,y)\,.
\endaligned
$$
It follows from the very definition of $Q$ that for every $w\in \C$ with $|w|=1$,
$$
\aligned
\overline{w}\,(\partial_\zeta Q)(\zeta,\eta)&=(\partial_\zeta Q)(w\zeta,\eta)=(\partial_\zeta Q)(w \zeta,w \eta)\,.
\endaligned
$$
Since $|k(x,y)|=1$, we can continue with
\begin{align*}
K_1=\int_{\wh\Omega\times\wh\Omega} &\left[(\partial_\zeta Q)\left(\hat f(y),\hat g(y)\right)-(\partial_\zeta Q)\left(k(x,y)\hat f(x),k(x,y)\hat g(x)\right)\right]\\
&\times\left[\hat f(y)-\hat f(x)k(x,y)\right]\,\wrt |m_T|(x,y)\,.
\end{align*}
In the same way we derive the analogue of the identity above for $K_2$:
\begin{align*}
K_2=\int_{\wh\Omega\times\wh\Omega} &\left[(\partial_\eta Q)\left(\hat f(y),\hat g(y)\right)-(\partial_\eta Q)\left(k(x,y)\hat f(x),k(x,y)\hat g(x)\right)\right]\\
&\times\left[\hat g(y)-\hat g(x)k(x,y)\right]\,\wrt |m_T|(x,y)\,.
\end{align*}
At this point we use \eqref{eq: sycn final 3} and apply
Corollary~\ref{l: sycn final stima} with
$$
\aligned
\alpha& =\hat f(y)-\hat f(x)k(x,y) &\hskip 30pt
\zeta & =\hat f(x)k(x,y)\\
\beta & =\hat g(y)-\hat g(x)k(x,y) &
\eta & =\hat g(x)k(x,y)\,.
\endaligned
$$
It follows that
\begin{align*}
\frac{2 I_2}{2\delta\cos\phi}&\geq\int_{\wh\Omega\times\wh\Omega}\mod{\hat f(y)-\hat f(x)k(x,y)}\mod{\hat g(y)-\hat g(x)k(x,y)}\wrt|m_T|(x,y)\\
&\geq \mod{\int_{\wh\Omega\times\wh\Omega}\left(\hat f(y)-\hat f(x)k(x,y)\right)\left(\overline{\hat g(y)}-\overline{\hat g(x)}\overline{k(x,y)}\right)\wrt|m_T|(x,y) }\\
&=2\mod{\int_{\wh\Omega\times\wh\Omega}\hat f(y)\overline{\hat g(y)}\wrt|m_T|(x,y)-\int_{\wh\Omega\times\wh\Omega}\hat f (x)\overline{\hat g(y)}\wrt m_T(x,y) }\\
&=2\mod{\int_\Omega \mathbf{T}(1)f\overline{g}\wrt\nu-\int_\Omega T(f)\overline{g}\wrt\nu}.
\end{align*}
We proved that
\begin{align}
\label{eq: sycn final 2}
I_2\geq 2\delta\cos\phi
\mod{\int_\Omega \left(\mathbf{T}(1)I-T\right)(f)\overline{g}\wrt\nu}.
\end{align}
The proposition now follows by combining \eqref{eq: sycn final 1} with \eqref{eq: sycn final 2}.
\end{proof}

\begin{remark}
Observe that in the preceding proof the crucial property of $Q$, summarized in Theorem~\ref{t: convex}, was applied through Corollary~\ref{l: sycn final stima}, and was needed in order to estimate the term $I_2$. That term was also where the Gelfand transform and Lemma \ref{l: kernel rep.} were applied.
\end{remark}

\begin{remark}
Suppose that $\nu(\Omega)<\infty$ and that $T$ is Markovian. Then $\mathbf{T}=T$, the measure $m_T$ is positive and symmetric and Lemma~\ref{l: kernel rep.} gives
$$
\int_\Omega(I-T)(u) v\wrt\nu=\int_{\wh\Omega\times\wh\Omega}\left(\int_{\{x,y\}}\cG\hat u\cdot\hat v\wrt\nu_{x,y}\right)\wrt m_T(x,y),
$$
for all $u,v\in L^\infty(\Omega,\nu)$. Therefore, in this case, Proposition~\ref{p: 4} immediately follows from Proposition~\ref{l: two-point}.
\end{remark}

As discussed at the beginning of this section, Theorem \ref{t: sycn bil embedding} (bilinear embedding) will follow once we confirm the next corollary.

\begin{corollary}
\label{c: sycn 1}
Let $p\in (2,\infty)$ and $\epsilon\in (0,1/2)$. Define $\delta$ and $Q$ as in \eqref{eq: def delta} and \eqref{eq: def bellman}, respectively. For all $f\in{\rm D}(\cA_p)$, $g\in {\rm D}(\cA_{q})$ and every $\phi\in [-\phi_{p_\epsilon},\phi_{p_\epsilon}]$, inequality \eqref{eq: heatflow p1 bis} holds with $B_0=2\delta\cos\phi$.
\end{corollary}
\begin{proof}
Since $(T_t)_{t>0}$ is a symmetric contraction semigroup and, by Proposition~\ref{p: 3}, $(\partial_\zeta Q)(f,g)\in L^q(\Omega,\nu)$ and $(\partial_\eta Q)(f,g)\in L^p(\Omega,\nu)$, the corollary follows from applying Proposition~\ref{p: 4} with $T_t$ in place of $T$,
dividing by $t$ both sides of the ensuing inequality and passing to the limit as $t\rightarrow0$.
\end{proof}

\begin{remark}
In order to apply the heat-flow argument of Section~\ref{s: heat-flow} we used Proposition~\ref{t: 1}.
However, Corollary~\ref{c: sycn 1} immediately gives a different proof of this result. Indeed, by Corollary~\ref{c: sycn 1} applied with $g=0$,
$$
 \Re \int_\Omega e^{i\phi}(\partial_\zeta Q)(f,0)\cA f\wrt\nu\geq 0
$$
for all $f\in {\rm D}(\cA_p)$ and $\phi\in[-\phi_{p_\epsilon},\phi_{p_\epsilon}]$. Since $(\partial_\zeta Q)(\zeta,0)=(p+2\delta)\overline{\zeta}|\zeta|^{p-2}$ for any $\zeta\in\C$, and $\phi_{p_\epsilon}\rightarrow \phi_p$ as $\epsilon\rightarrow0$, the inequality above can be rewritten as the sectorial inequality
$$
\mod{\Im\int_\Omega \cA_pf\cdot\overline{f}|f|^{p-2}\wrt\nu}\leq \cot\phi_p\cdot\Re\int_\Omega \cA_pf\cdot\overline{f}|f|^{p-2}\wrt\nu\,,
$$
which is well known to be equivalent to Proposition~\ref{t: 1}.\hfill\qed
\end{remark}

\subsection*{Acknowledgements}
The first author was partially supported by the Gruppo Nazionale per l'Analisi Matematica, la Probabilit\`a e le loro Applicazioni (GNAMPA) of the Istituto Nazionale di Alta Matematica (INdAM). The second author was partially supported by the Ministry of Higher Education, Science and Technology of Slovenia (research program Analysis and Geometry, contract no. P1-0291).
A part of this paper was written during the first author's visit at Faculty of Mathematics and Physics, University of Ljubljana, and during the second author's visit at Dipartimento di matematica, University of Genova. The authors wish to thank the said institutions for their kind hospitality.

\end{document}